\documentclass[11pt,a4paper,english]{article}

\usepackage[T1]{fontenc}
\usepackage[utf8]{inputenc}
\usepackage[final]{microtype}  
\usepackage{lmodern}

\usepackage{graphicx}        
\graphicspath{{figures/}}

\usepackage[a4paper, hmargin={2.7cm,2.7cm},vmargin={3.3cm,3.3cm}]{geometry}
\usepackage{mathrsfs}
\usepackage{amsmath, amssymb, amsxtra}
\usepackage{amsthm, bm}
\usepackage{mathtools, mathdots}
\usepackage[normalem]{ulem}

\usepackage{listings}

\usepackage{babel}
\usepackage{cite}
\usepackage{paralist}  
\usepackage{url}
\usepackage{twoopt}
\usepackage{float}   

\usepackage{caption}
\usepackage[margin=0.2cm]{subcaption} 

\usepackage{siunitx}
\usepackage{tikz}

\usepackage[bf,compact,pagestyles,small]{titlesec}
\titlespacing*{\section}{0pt}{14pt}{4pt}
\titlespacing*{\subsection}{0pt}{8pt}{3pt}

\usepackage{etoolbox}
\makeatletter
\patchcmd{\ttlh@hang}{\parindent\z@}{\parindent\z@\leavevmode}{}{}
\patchcmd{\ttlh@hang}{\noindent}{}{}{}
\makeatother

\usepackage{pgfplotstable}
\usepackage{pgfplots}
\pgfplotsset{compat=newest}

\pgfplotsset{every axis/.append style={
                    axis x line=middle,
                    axis y line=middle,
                    axis line style={->},
                    xlabel style={at={(ticklabel* cs:1)},anchor=north west},
                 }}

\usepackage{fancyhdr,lastpage}
\def\maketimestamp{\count255=\time
\divide\count255 by 60\relax
\edef\thetime{\the\count255:}%
\multiply\count255 by-60\relax
\advance\count255 by\time
\edef\thetime{\thetime\ifnum\count255<10 0\fi\the\count255}
\edef\thedate{\number\day-\ifcase\month\or Jan\or Feb\or Mar\or
             Apr\or May\or Jun\or Jul\or Aug\or Sep\or Oct\or
             Nov\or Dec\fi-\number\year}
\def\timstamp{\hbox to\hsize{\tt\hfil\thedate\hfil\thetime\hfil}}}
\maketimestamp
\fancypagestyle{plain}{%
\fancyhf{} 
\fancyfoot[C]{\small \thepage{} of \pageref{LastPage}}}

\numberwithin{equation}{section}  
\allowdisplaybreaks[4]

\newtheorem{theorem}{Theorem}[section]

\newtheorem{lemma}[theorem]{Lemma}

\newtheorem{corollary}[theorem]{Corollary}
\theoremstyle{definition}

\theoremstyle{remark}

\newtheorem{remark}[theorem]{Remark}

\DeclareMathOperator{\Span}{span} %
\DeclareMathOperator*{\esssup}{ess\,sup} %
\DeclareMathOperator{\ft}{\mathcal{F}}
\DeclareMathOperator{\exponential}{e}

\newcommandtwoopt{\mixedS}[2][\cG][\cH]{S_{{#1},{#2}}} 

\newcommandtwoopt{\gaborG}[3][\Lambda][\Gamma]{\mathcal{G}(#3,#1,#2)} 

\newcommand{\myexp}[1]{\exponential^{#1}}

\newcommand*{\numbersys}[1]{\ensuremath{\mathbb{#1}}}

\newcommand*{\R}{\numbersys{R}}

\newcommand*{\Q}{\numbersys{Q}}

\newcommand*{\Z}{\numbersys{Z}}

\newcommand*{\cH}{\mathcal{H}}

\newcommand*{\cG}{\mathcal{G}}

\newcommand{\itvoc}[2]{\ensuremath{\left({#1},{#2}\right]}} 
\newcommand{\itvoo}[2]{\ensuremath{\left({#1},{#2}\right)}} %
\newcommand{\itvcc}[2]{\ensuremath{\left[{#1},{#2}\right]}} %
\newcommand{\itvco}[2]{\ensuremath{\left[{#1},{#2}\right)}} %
\newcommand{\itvccs}[2]{\ensuremath{\lbrack{#1},{#2}\rbrack}} %
\newcommand{\itvcos}[2]{\ensuremath{\lbrack{#1},{#2})}} %
\newcommand{\abs}[1]{\ensuremath{\left\lvert#1\right\rvert}}

\newcommand{\absbig}[1]{\ensuremath{\bigl\lvert#1\bigr\rvert}}

\newcommand{\norm}[2][]{\ensuremath{\left\lVert#2\right\rVert_{#1}}}

\newcommand{\innerprod}[3][]{\ensuremath{\left\langle #2,#3\right\rangle_{\! #1}}}

\newcommand{\set}[1]{\ensuremath{\left\lbrace{#1}\right\rbrace}}
\newcommand{\seq}[1]{\ensuremath{\left\lbrace{#1}\right\rbrace}}
\newcommand{\setprop}[2]{\ensuremath{\left\lbrace{#1} : {#2}\right\rbrace}}

\newcommand{\ceil}[1]{\left\lceil #1 \right\rceil}

\newcommand{\floor}[1]{\left\lfloor #1 \right\rfloor}

\newcommand*{\frameset}{\mathscr{F}}

\usepackage{xspace}

\newcommand*\oline[1]{%
  \vbox{%
    \hrule height 0.5pt
    \kern0.25ex
    \hbox{%
      \kern-0.1em
      \ifmmode#1\else\ensuremath{#1}\fi
      \kern-0.1em
    }
  }
}

\newcommand{\rpar}{r}

\usepackage{hyperref}
\hypersetup{
    colorlinks,
    linkcolor={red!50!black},
    citecolor={blue!50!black},
    urlcolor={blue!80!black},
 pdfview={FitH},
 pdfstartview={FitH},
 pdfauthor = {Jakob Lemvig}
 pdftitle = {Lectures Notes},
 pdfkeywords = {frame, frame set, Gabor system, Hermite},
 pdfcreator = {LaTeX with hyperref package},
 pdfproducer = PDFlatex}

\makeatletter
\def\blfootnote{\xdef\@thefnmark{}\@footnotetext} 
\def\subjclass{\xdef\@thefnmark{}\@footnotetext}
\long\def\symbolfootnote[#1]#2{\begingroup%
\def\thefootnote{\fnsymbol{footnote}}\footnote[#1]{#2}\endgroup} 
\if@titlepage
  \renewenvironment{abstract}{%
      \titlepage
      \null\vfil
      \@beginparpenalty\@lowpenalty
      \begin{center}%
        \bfseries \abstractname
        \@endparpenalty\@M
      \end{center}}%
     {\par\vfil\null\endtitlepage}
\else
  \renewenvironment{abstract}{%
      \if@twocolumn
        \section*{\abstractname}%
      \else
        \small
        \list{}{%
          \settowidth{\labelwidth}{\textbf{\abstractname:}}
          \setlength{\leftmargin}{50pt}
          \setlength{\rightmargin}{50pt}
          \setlength{\itemindent}{\labelwidth}
          \addtolength{\itemindent}{\labelsep}
        }
        \item[\textbf{\abstractname:}]

      \fi}
      {\if@twocolumn\else\endlist\fi}
\fi
\makeatother

\begin{document}

\title{On the non-frame property of
  Gabor systems with Hermite
  generators and the frame set conjecture}

\date{\today}

 \author{Andreas Horst\footnote{Technical University of Denmark, E-mail:
     \protect\url{ahor@dtu.dk}}\phantom{$\ast$}, Jakob Lemvig\footnote{Technical
     University of Denmark, Department of Applied Mathematics and Computer
     Science, Matematiktorvet 303B, 2800 Kgs.\ Lyngby, Denmark, ORCID:
     0000-0002-9338-7755, E-mail:
     \protect\url{jakle@dtu.dk}}\phantom{$\ast$},
Allan Erlang Videb\ae k\footnote{Veo Technologies, E-mail:
     \protect\url{allanerlang@gmail.com}}}

 \blfootnote{2010 {\it Mathematics Subject Classification.} Primary
   42C15. Secondary: 42A60}
 \blfootnote{{\it Key words and phrases.} Frame, frame set, Gabor
   system, Hermite functions, Zibulski-Zeevi matrix}

\maketitle

\thispagestyle{plain}
\begin{abstract}
  The frame set conjecture for Hermite functions formulated in
  \cite{Groechenigmystery2014} states that the Gabor frame set
  for these generators is the largest possible, that is, the time-frequency
  shifts of the Hermite functions associated with sampling rates
  $\alpha$ and modulation rates $\beta$ that avoid all known obstructions lead
  to Gabor frames for $L^{2}(\mathbb{R})$. By results in \cite{LyubarskiuiFrames1992,SeipDensity1992} and
  \cite{Lemvigsome2017}, it is known that the conjecture is true for the
  Gaussian, the $0$th order Hermite functions, and false for Hermite functions
  of order $2,3,6,7,10,11,\dots$, respectively. In this paper we disprove the
  remaining cases \emph{except} for the $1$st order Hermite function.
\end{abstract}

\section{Introduction}
\label{sec:introduction}
Given a function $g \in L^{2}(\R)$ and two positive parameters $\alpha$ and $\beta$,
the set of functions
$\mathcal{G}(g,\alpha,\beta):=\set{\myexp{2\pi i \beta m \cdot}g(\cdot-\alpha k)}_{k,m\in \Z}$
is said to be a Gabor frame for $L^2(\R)$ if there exist constants $A,B>0$, called
frame bounds, such that
\begin{equation*}
A \norm{f}^2 \le \sum_{k,m \in \Z} \abs{\innerprod{f}{\myexp{2\pi i
      \beta m
      \cdot}g(\cdot- \alpha k)}}^2 \le B \norm{f}^2 \quad \text{for all } f
\in L^2(\R). 
\end{equation*}
We refer to \cite{Christensenintroduction2016,MR1843717}
for an introduction to frames and Gabor analysis.

The Gabor frame set, or simply the \emph{frame set}, of a window function
$g \in L^2(\R)$, denoted by $\frameset{(g)}$, is the set of tuples of sampling
and modulation parameters $(\alpha,\beta)\in\mathbb{R}_{>0}^2$ for which the
associated Gabor system $\mathcal{G}(g,\alpha,\beta)$ is a frame for $L^2(\R)$.
The Gabor frame set conjecture for Hermite
functions~\cite{Groechenigmystery2014} states that the frame set for Hermite
functions of even orders is $\{(\alpha,\beta) \in \R^2_{>0}: \alpha\beta <1\}$
and for odd orders is $\{(\alpha,\beta) \in \R^2_{>0}: \alpha\beta <1 \text{ and
} \alpha\beta \neq 1/2, 2/3, \dots\}$, where the Hermite functions is given by
\begin{equation}
    \label{eq:def_hn}
    h_n(x) = (-1)^n (c_n)^{-1/2} \myexp{\pi x^2} \left(\frac{d^n}{dx^n} \myexp{-2\pi x^2}\right)
\end{equation}
for $c_n := (2\pi)^n 2^{n-1/2} n!$ for $n \in \Z_{\ge 0}$. The frame set
conjecture is true for the Gaussian case $n=0$ as proved by
Lyubarskii~\cite{LyubarskiuiFrames1992} and Seip and
Wallst\'en~\cite{SeipDensity1992,SeipDensity1992a}, but false for orders
$n=4m+2$ and $n=4m+3$ for all non-negative integers $m \in \Z_{>0}$ as proved by
the second-named author~\cite{Lemvigsome2017}. Numerical simulations in
\cite{Lemvigsome2017} suggest that the conjecture is also false for the two
cases $n=4$ and $n=5$, but no proof is provided, and the cases
$n=1,4,5,8,9,12,13,\dots$, i.e., $n=4m \, (m>0)$ and $n=4m+1 \, (m \ge 0)$ are
still open. The goal of this work is to disprove the remaining open cases $n=4m$
and $n=4m+1$, where $m \in \Z_{>0}$, except $n=1$. This will, in turn, show that
the frame set conjecture for Hermite functions is false for all orders $n \ge 2$
and that the frame set for Hermite functions with two or more zeros necessarily
are more mysterious than originally believed. We remark that we do not shed new
light on the frame set conjecture for the first order Hermite function $h_1$ due
to Lyubarskii and Nes~\cite{MR3027914}.

One insight from~\cite{Lemvigsome2017} was to split the Gabor frame set problem
for Hermite functions into four subproblems depending on the eigenvalue of the
window function with respect to the Fourier transform on $L^2(\R)$. Recall that
the Hermite functions $\setprop{h_{4m+\ell}}{m \in \Z_{\ge 0}}$ has eigenvalue
$\lambda=(-i)^\ell$ for integers $\ell = 0,1,2,3$ with respect to the Fourier
transform. The methods used in \cite{Lemvigsome2017} to disprove the frame set
conjecture for Hermite functions of order ${4m+2}$ and ${4m+3}$,
$m \in \Z_{\ge 0}$, are not specific to the Hermite functions, but works for any
sufficiently nice eigenfunction of the Fourier transform with eigenvalue either
$\lambda=-1$ or $\lambda=i$. However, these methods cannot be modified to also
disprove the remaining cases. This is obvious from the fact that the methods
only rely on eigenvalues with respect to the Fourier transform of the window:
since $h_0$ and $h_{4m}$, $m \in \Z_{>0}$, have the same eigenvalue
($\lambda=1$), and the frame set conjecture is true for the Gaussion $h_0$, one
cannot use methods only relying on the eigenvalue to disprove the conjecture for
$h_{4m}$, $m \in \Z_{>0}$. The methods developed in this work are specific to
the Hermite functions and relies crucially on a number of properties of the
Hermite functions, in particular, on the existence and approximate location of
positive zeros.

We will, in fact, give many counterexamples to each conjecture and, similar to
\cite{Lemvigsome2017}, the counterexamples appears on hyperbolas
$\alpha \beta=1/2, \alpha \beta=1/3, \alpha \beta=1/4$ and $\alpha \beta=2/3$.
Gabor systems $\mathcal{G}(g,\alpha,\beta)$ with $\alpha\beta \in \Q$ are called
rationally oversampled systems, and their frame property can be completely
characterized by the Zak transform and the Zibulski-Zeevi
matrix~\cite{ZibulskiAnalysis1997}. Thus, it is no surprise that our
counterexamples are based on properties of the Zak transform of Hermite
functions. However, the way we will study these properties is non-standard in
Gabor analysis as we will fix the time and frequency variable of the Zak
transform and consider the modular parameter as a variable. Moreover, in
\cite{Lemvigsome2017} the location $(\alpha,\beta)$ of the counterexamples for
all sufficiently nice functions in the two eigenspaces of the Fourier transform
was fixed and $\alpha,\beta \in \itvcc{0.5}{1.16}$. In this work, contrary to
\cite{Lemvigsome2017}, the location $(\alpha, \beta)$ of the counterexamples on
the hyperbola depends on the order $n$ of the Hermite function $h_{n}$, and we
show that $\alpha$ and $\beta$ grow as $n^{1/2}$ and $n^{-1/2}$ and vice versa
(up to specified constants). Our techniques provide new obstructions of the
frame property, not only for the open cases, but for all orders $n \ge 3$. We
need to restrict our attention to orders greater than or equal to three as our
methods rely on at least two non-negative zeros of the window function.

\subsection{Outline of the paper}
In Section~\ref{sec:hermite-functions} we improve on a lower bound of the
largest zero of the Hermite functions by Szeg\"o~\cite{MR0106295}. In
Section~\Ref{sec:ZakTransform} we introduce the Zak transform $Z_{\lambda}$ as a
unitary transform of $L^{2}(\R)$ onto $L^{2}(\itvco{0}{1}^{2})$ and explain its
role in the frame set conjecture. In Section~\ref{sec:zeros-symmetries-zak} we
turn to zeros and symmetries of the Zak transform. We recall some known results
on even and odd functions in Section~\ref{sec:symmetries-va}, and, in
Section~\ref{sec:obstr-frame-prop}, how zeros of the Zak transform lead to
obstructions of the frame property. In Section~\ref{sec:symmetries-par} we prove
a symmetry property of the auxiliary function
$\lambda \mapsto Z_{s\lambda}h_{n}\bigl(\frac{x_0}{s^{2}},\gamma_0\bigr)$ for
certain fixed values of $x_0$ and $\gamma_0$ with $s^2=2,3,4$. More precisely,
we will prove that the same symmetry property will hold for \emph{one}
$\gamma_0$ value (either $0$ or $1/2$) and for $s^2$ different $x_0$-values each
separated by $1/s^2$. The symmetry property can be seen as a \emph{pointwise}
form of the modular characteristics in the sense of theta functions and states
(see Theorem~\Ref{thm:symm_modu_even_n} and~\Ref{thm:symm_modu_odd_n}) that
$\kappa \mapsto Z_{s2^\kappa}h_{n}\bigl(\frac{x_0}{s^{2}},\gamma_0\bigr)$ is,
depending on the value of $n$, an odd or even (continuous) function on $\R$. The
auxiliary function is of interest since any of its zeros lead to obstructions of
the frame properties of Gabor systems generated by $h_{n}$. Indeed, the known
counterexamples from \cite{Lemvigsome2017} for the frame set conjecture for
Hermite functions of order ${4m+2}$ and ${4m+3}$, $m \in \Z_{\ge 0}$, follow
immediately from these properties. To disprove the remaining cases we need, for
any $x_{0} \in [-1/4,1/4]$, the existence of zeros of
$\kappa \mapsto Z_{s2^\kappa}h_{n}\bigl(\frac{x_0}{s^{2}},\gamma_0\bigr)$, where
$\gamma_{0}$ is $0$ for odd $n$ and $1/2$ for even $n$, and $n \ge 3$. This
existence is proved in Section~\ref{sec:additional-zeros-zak}. Finally, in
Section~\ref{sec:counterexamples} we give the counterexamples of the frame set
conjecture for Hermite functions.

\subsection{Related works}

The study of Gabor systems generated by Hermite functions is closely connected
to coherent states associated with higher Landau levels. For example, Abreu et
al.~\cite{MR3424534} demonstrate how properties of Gabor systems with Hermite
windows can be identified with aspects of quantum mechanics, specifically the
behavior of a charged particle in a constant homogeneous magnetic field.
Similarly, the results presented in this paper illustrate instances of dense
superpositions of generating states in higher Landau levels, where the energy
can be made arbitrarily small. One celebrated illustration of this interplay is
the Quantum Hall Effect (QHE), which has led to several Nobel Prizes in physics
and chemistry since the experimental discovery of the \emph{integer} QHE by
von~Klitzing in 1980. The \emph{integer} QHE can be explained by the formation
of Landau levels, where the integers correspond to the order of Hermite
functions. Notably, von~Klitzing~\cite{WOS:A1986D812000001} used the wave
function proposed by Laughlin~\cite{Laughlin1981}, as explicitly stated in
equation (5) of these seminal works, which corresponds precisely to a Gabor
system with Hermite function generators. For a deeper exploration of these
connections, we refer to \cite{WOS:000235953900002,MR3203099, MR3424534} and the
references therein.

Due to the close relationship between quantum states of the quantum harmonic
oscillator and Gabor analysis with Hermite functions, similar ideas have emerged
in both fields. The first study of the zeros of the Zak transform of Hermite
functions appeared in a physics context, in the work by Boon, Zak, and Zucker
\cite{WOS:A1983QC19800014,WOS:A1981LU63700026}\footnote{These papers have
  largely been unknown to the frame theory community; see \cite{MR1757401} for a
  rare exception. We were made aware of the connection between our work and
  these papers by one of the referees.}. On page 320 of
\cite{WOS:A1983QC19800014}, the authors show that for rational values of
$\alpha\beta < 1$, the Gabor system $\mathcal{G}(h_{n},\alpha,\beta)$ is
complete in $L^2(\mathbb{R})$. This result was recently rediscovered and
extended to a larger class of window functions in
\cite{GroechenigCompleteness2016}. Additionally, Boon, Zak, and Zucker
\cite{WOS:A1983QC19800014} proved results regarding the zeros of the Zak
transform of Hermite functions, which were later rediscovered and generalized by
the second author \cite{Lemvigsome2017}. The counterexamples to the Gabor frame
set conjecture for Hermite functions in \cite{Lemvigsome2017} can be derived
from the results in \cite{WOS:A1983QC19800014} combined with the Zibulski-Zeevi
characterization of Gabor frames~\cite{ZibulskiAnalysis1997}. We note that the
current work does not overlap with
\cite{WOS:A1983QC19800014,WOS:A1981LU63700026}; see Remark~\ref{rem:four-zeros}.

We conclude the introduction with a brief summary of positive results towards
characterizing the Gabor frame set of Hermite functions. It is well-known that
the Gabor system $\mathcal{G}(h_{n},\alpha,\beta)$ satisfies the upper frame
bound and, as mentioned above, it is complete in $L^2(\mathbb{R})$ for rational
$\alpha\beta \le 1$. Gröchenig and
Lyubarskii~\cite{GroechenigGabor2009,GroechenigGabor2007a} showed that the Gabor
system $\mathcal{G}(h_{n},\alpha,\beta)$ forms a frame if
$\alpha\beta < \frac{1}{n+1}$. More recently, Ghosh and
Selvan~\cite{ghosh2023gabor} conducted a numerical study of the frame set of
Hermite functions, as well as other functions, using a connection to sampling
theory in shift-invariant spaces.

\section{Hermite functions}
\label{sec:hermite-functions}

Hermite functions arise in many different
contexts, e.g., as eigenfunctions of both the Hermite operator
$H=-\frac{d^2}{dx^2}+(2\pi x)^2$ and the Fourier transform: 
\[ 
 \hat{h}_n(\gamma) = (-i)^n h_n(\gamma)  \quad a.e.\ \gamma \in \R.
\] 
Here, the Fourier transform is defined for $f \in L^1(\R)$ by 
\[
\ft
f(\gamma)=\hat f(\gamma) = \int_{\R} f(x)\myexp{-2 \pi i
  \gamma x} \mathrm{d}x
  \]
with the usual extension to $L^2(\R)$. We let
$E_\ell = \overline{\Span}\setprop{h_{4m+\ell}}{m \in \Z_{\ge 0}}$ $\subset L^2(\R)$, $\ell=0,1,2,3$,
denote the eigenspace of the Fourier transform corresponding to the
eigenvalue $(-i)^\ell$.

Since the Fourier transform is a unitary operator, it preserves the
frame property. Moreover, since the Fourier transform switches the role of the sampling
and modulation parameter, we see that the system $\gaborG[\alpha][\beta]{g}$ is a frame if and
only if the Fourier transform of the system $\gaborG[\beta][\alpha]{\hat{g}}$
is a frame. As a consequence, we immediately have the following simple, but useful
result showing that $\frameset{(h_{n})}$ is symmetric about the line $\alpha=\beta$.
\begin{lemma}
  \label{lem:frame-set-symmetry-h_n}
  Let $\alpha,\beta>0$ and $A,B>0$. The Gabor system $\gaborG[\alpha][\beta]{h_{n}}$
  is a frame with bounds $A$ and $B$ if, and only if,  the Gabor system
  $\gaborG[\beta][\alpha]{h_{n}}$ is a frame with bounds $A$ and $B$.
\end{lemma}

\subsection{The Hermite polynomials and location of their zeros}

The Hermite functions $h_n$ defined in~\eqref{eq:def_hn} and the Hermite polynomials $H_n$ defined by
\[
  H_{n}(x)=(-1)^{n}e^{x^{2}}{\frac {d^{n}}{dx^{n}}}e^{-x^{2}}, \qquad
  n \in \Z_{\ge 0},
\]
are related by $h_n(x)= d_n H_n(\sqrt{2\pi}x) \myexp{-\pi x^2}$, where $d_n$ is a
positive constant. A few of the Hermite functions $h_{n}$ of interest in this
work are plotted in Figure~\ref{fig:hermite-funs}. In particular, the zeros of the Hermite functions
can be determined by a simple scaling of the roots of the Hermite polynomials.
\begin{figure}
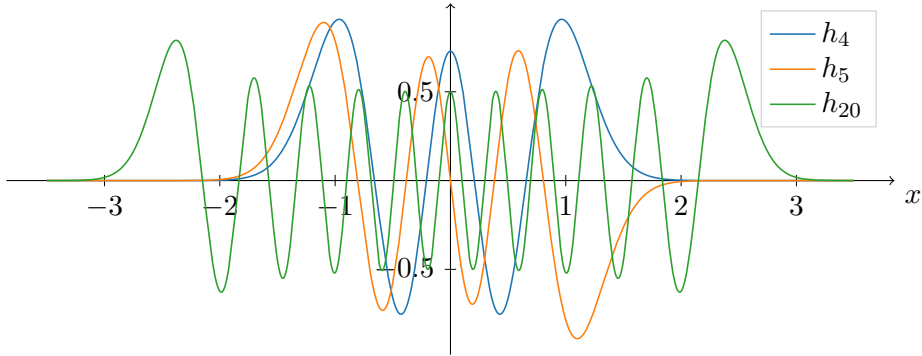

  \centering



\caption{The graph of $h_4$, $h_5$, and $h_{20}$.}
\label{fig:hermite-funs}
\end{figure}

Since $\setprop{H_n(x)}{n=1,2,\dots}$ are orthogonal polynomials with respect
to the Gaussian weight $\myexp{-x^2}$, the Hermite polynomial $H_n(x)$ has $n$
real and simple roots. We denote these roots by $x_1,x_2, \dots, x_n$ in
descending order so that $x_n < x_{n-1} < \dots < x_1$. It is well-known that
all zeros belong to the interval $\itvcc{-\sqrt{2n+1}}{\sqrt{2n+1}}$.
By symmetry of the weight function, the Hermite polynomials $H_n(x)$ are even
and odd functions for $n$ even and odd, respectively. It follows that
$x_k = -x_{n-k}$ for all $k=1,2,\dots,n$, and we therefore restrict our
attention to zeros on the positive real line.

The properties of zeros of the classical orthogonal polynomials are of interest
in many different areas of mathematics, e.g., in quadrature formulas, but also
in applications such as spherical designs \cite{MR3151649}. One classical and powerful tool to
study zeros of orthogonal polynomials is Sturm's comparison theorems for second
order differential equations~\cite[Thm. 1.82.2]{MR0106295}. Using this theorem, it can be shown, see \cite[(6.31.15)]{MR0106295}, that the
zeros are convex in the sense that the distance between two consecutive roots
increases as we move away from the origo $x=0$. To be precise, for three
consecutive zeros $x_{k+1} < x_{k} < x_{k-1}$, where $1< k \le \floor{n/2}$, we
have $x_{k} - x_{k+1} < x_{k-1} - x_{k}$.

The location of the zeros is often formulated as asymptotic estimates for
$n \to \infty$, however, we will need bounds on the location of the roots that
also hold for small values of $n$. Some bounds are well-known, but mainly upper
bounds and mainly of $x_1$. We will need explicit lower bounds only depending on
$n$ and $k$. Szeg\"o~\cite{MR0106295} proves using Laguerre's Theorem on the
roots of polynomials the lower bound $x_1 > \sqrt{(n-1)/2}$ on the largest root,
see page 130 in \cite[(6.32.6)]{MR0106295}. We will need a sharper bound to avoid
handling a number of special cases for small $n$. Its proof is simple and
only relies on the convexity of roots and the well-known fact that the square of
the Hermite roots sum to $\frac{n(n-1)}{2}$.

     \begin{lemma}
     \label{lem:x1-simple-bound}
Let $n \ge 2$ be an integer. The largest root $x_1$ of $H_n(x)$ satisfy the lower bound
    \begin{equation}
        \label{eq:x1-simple-lower-bound}
        x_1 > \sqrt{3/2} \frac{n-1}{\sqrt{n+1}}.
    \end{equation}
\end{lemma}
\begin{proof}
  Let $n \ge 2$. We divide the interval $\itvcc{x_{n}}{x_{1}}$, that is, $\itvcc{-x_{1}}{x_{1}}$, into $n-1$
  uniform intervals of length $2x_{1}/(n-1)$. By the convexity of the roots, we
  then have
  \begin{equation}
    x_{k} < x_{1} - (k-1)\frac{2x_{1}}{n-1} = \frac{n-2k+1}{n-1} x_{1}
    \label{eq:root-convexity-estimate}
  \end{equation}
  for $k=1,\dots,\floor{n/2}$. Thus, by~\eqref{eq:root-convexity-estimate}, we
  can estimate:
  \begin{align*}
    \sum_{k=1}^{\floor{n/2}} x_{k}^{2} < \frac{x_{1}^{2}}{(n-1)^{2}} \sum_{k=1}^{\floor{n/2}} (n-2k+1)^{2} = \frac{x_{1}^{2}}{(n-1)^{2}} \frac{n}{6} (n^{2}-1) = x_{1}^{2}\, \frac{n(n+1)}{6(n-1)}
  \end{align*}
  Now, using that the positive zeros of $H_{n}(x)$ satisfy, see \cite[eq.~(5.5.4) and p.~142]{MR0106295}, the relation
  \[
    \sum_{k=1}^{\floor{n/2}} x_{k}^{2} = \frac{n(n-1)}{4},
  \]
  we arrive at~\eqref{eq:x1-simple-lower-bound}.
\end{proof}

The lower bound \eqref{eq:x1-simple-lower-bound} improves on
$x_1 >  \sqrt{(n-1)/2}$ for all $n\ge 2$ and asymptotically by a factor of $\sqrt{3}$.
The upper bound of $x_{k}$ can be improved to be
$x_{k} \le \sqrt{2n - 2} \cos \tfrac{(k-1)\pi}{n-1}$, cf.~\cite{MR1858270}. We refer to \cite{MR1909981} and the references therein for a
survey on bounds of roots of Hermite polynomials.

\subsection{Properties of the Hermite functions}
\label{sec:prop-herm-funct}

Since $h_n(x)= d_n H_n(\sqrt{2\pi}x) \myexp{-\pi x^2}$, the zeros of $h_n$ are given by
$\frac{1}{\sqrt{2\pi}}x_k$, $k=1,2,\dots,n$, and they lie in the
oscillatory region
$\itvcc{-\frac{\sqrt{2n+1}}{\sqrt{2\pi}}}{\frac{\sqrt{2n+1}}{\sqrt{2\pi}}}$,
where $\frac{\sqrt{2n+1}}{\sqrt{2\pi}}$ is the turning point of the
harmonic oscillator in quantum mechanics
\[
  h^{''}_n(x) + (2n+1-x^2/(2 \pi)) h_n(x) = 0, \quad x \in \R.
\] 
For Hermite polynomial $H_n(x)$ the coefficient of $x^n$ is positive so
$h_n(x)>0$ for $x > \frac{1}{\sqrt{2\pi}}\sqrt{2n+1}$. Since $h_{n}$ satisfies
the above differential equation, it follows that $h_{n}$ is convex, i.e.,
$h''_{n}(x) > 0 $, for $\abs{x} > \frac{1}{\sqrt{2\pi}}\sqrt{2n+1}$, and thus
monotonically decreasing on $\itvoo{\frac{1}{\sqrt{2\pi}}\sqrt{2n+1}}{\infty}$.

\section{The Zak transform in Gabor analysis}
\label{sec:ZakTransform}

In this section we study a classical transform that has been used by Weil~\cite{MR0165033} in harmonic
analysis on locally compact abelian groups, by
Gel'fand~\cite{MR0073136} in the study of Schr\"odinger's equation,
and by Zak~\cite{MR1478343} in solid state physics. In Gabor analysis, see
\cite{MR3218799,MR3393698,MR947891},  it is common to use the name
\emph{Zak transform}, and we follow this tradition. For any $\lambda>0$ , the Zak transform of a function $f \in L^2(\R)$ is defined as
\begin{equation}\label{eq:zakTransform}
\left(Z_{\lambda}f\right)(x,\gamma)
= \sqrt{\lambda}\sum_{k\in\mathbb{Z}} f(\lambda(x+
  k))\myexp{-2\pi i k \gamma}, \quad a.e.\ x, \gamma \in \mathbb{R},
\end{equation}
with convergence in  $L^2_\mathrm{loc}(\R)$; see \cite[Lemma 8.2.1]{MR1843717}. The Zak transform
$Z_\lambda$ is a unitary transform from $L^2(\R)$ to $L^2(\itvco{0}{1}^2)$, cf. \cite[Theorem 8.2.3]{MR1843717}, with
the following quasi-periodicity:
\begin{equation}
Z_\lambda f(x+1,\gamma)= \myexp{2\pi i\gamma} Z_\lambda f(x,
\gamma),  \quad  Z_\lambda f(x, \gamma +1) = Z_\lambda f(x, \gamma) \label{eq:Zak-quasi-periodicity}
\end{equation}
for a.e.~$x,\gamma \in \R$.

We will study the Zak transform of odd and even, sufficiently nice, functions.
By ``sufficiently nice'' we usually mean membership of the Wiener space $W(\R)$ of
functions $g \in L^\infty(\R)$ for which
$\sum_{k \in \Z} \esssup_{x \in \itvcc{0}{1}}\abs{g(x+k)}<\infty$.
E.g., if $f$ belongs to $W(\R)$ and is continuous, then $Z_\lambda f$ is
continuous, \cite[Lemma 8.2.1]{MR1843717}, hence a.e.~identities of the Zak transform will hold pointwise everywhere.
Under the stronger assumption $f \in W(\R)$ and $\hat{f} \in W(\R)$, we have
\begin{equation}
  \label{eq:F-of-Zak}
  Z_\lambda f(x,\gamma) = \myexp{2\pi i x \gamma}
  Z_{1/\lambda}\hat{f}(\gamma,-x) \quad \text{for all } x,\gamma \in \R,
\end{equation}
with absolute convergence of the series. Equation~\eqref{eq:F-of-Zak} is a consequence of
Poisson summation formula, see e.g., \cite{MR947891} or \cite[Proposition
8.2.2]{MR1843717}. Note that Hermite functions and, more generally, any function $f$ in
$E_\ell \cap W(\R)$ for $\ell=0,1,2,3$, satisfy the assumption $f, \hat{f} \in W(\R)$.

In Section~\ref{sec:symmetries-par}, we use the Poisson summation formula to reveal symmetries in the Zak transform of Hermite functions. To illustrate, we present a simple example of these pointwise modular characteristics, showing how they lead to zeros of the Zak transform. For $f \in E_\ell \cap W(\R)$ with $\ell=0,1,2,3$ and using $\hat{f} = (-i)^\ell f$, Poisson summation formula~\eqref{eq:F-of-Zak} gives
\begin{align}
    Z_\lambda f(0,0) = (-i)^\ell Z_{1/\lambda}f(0,0)
    \label{eq:Zak-1-symm-origo}
\intertext{and}
   Z_\lambda f(\tfrac{1}{2},\tfrac{1}{2}) = (-i)^{\ell-1} Z_{1/\lambda}f(\tfrac{1}{2},\tfrac{1}{2}),
   \label{eq:Zak-1-symm-mid}
\end{align}
where $\lambda > 0$. For $\ell=2$ and $\lambda=1$ this simplifies to $Z_1 f(0,0) = - Z_{1}f(0,0)$ which implies the following zero of the Zak transform
\begin{equation}
    Z_1 f(0,0) = 0 \quad \text{for } f \in E_2 \cap W(\R).
    \label{eq:Zak-1-zero-origo}
\end{equation}
Similarly, for $\ell =3$, we find 
\begin{equation}
Z_1 f(\tfrac{1}{2},\tfrac{1}{2}) = 0 \quad \text{for } f \in E_3 \cap W(\R).
\label{eq:Zak-1-zero-mid}
\end{equation}
These zeros were discovered by Boon, Zak and Zucker, see equation (26) and (27) in \cite{WOS:A1983QC19800014}, for Hermite functions $h_n$ of order $n=4m+2$ and $n=4m+3$ ($m \in \Z_{\ge 0}$), respectively.

\subsection{Rationally oversampled Gabor systems}
\label{sec:rati-overs-gabor}

The frame property of rationally oversampled Gabor systems, i.e., $\mathcal{G}(g,\alpha,\beta)$ with
\[
\alpha\beta \in \mathbb{Q}, \quad \alpha\beta=\frac{p}{q}<1 \quad \gcd(p,q)=1,
\]
can be characterized by the Zak transform in terms of the so-called
Zibulski-Zeevi matrix~\cite{ZibulskiAnalysis1997}. This matrix is a $p \times q$ matrix whose $(k, t)$-entry is given by
\begin{equation*}
p^{-\frac{1}{2}}(Z_{\frac{1}{\beta}}g)\left(x-t\frac{p}{q},
\gamma+\frac{k}{p}\right) \quad\text{for } a.e.
\; x,\gamma\in\ensuremath{\mathbb{R}}.
\end{equation*}
Indeed, the lower and upper frame bounds of $\mathcal{G}(g,\alpha,\beta)$ correspond
  to the smallest and largest singular values of the Zibulski-Zeevi matrix
  uniformly over $(x,\gamma) \in  \itvcos{0}{1}^2$. In case the
  Zibulski-Zeevi matrix contains a zero row, the smallest singular value becomes
  zero whereby the lower frame bound of $\mathcal{G}(g,\alpha,\beta)$ fails. The precise
  statement is the following result, on which all our counterexamples are based.

  \begin{lemma}
    \label{lem:zero-row-failure}
      Let $g \in W(\R)$ be continuous. Suppose $\alpha\beta = \frac{p}{q}\in \Q$ with $p,q$
      relatively prime. If
\[
Z_{\frac{1}{\beta}}g(x_0+\tfrac{t}{q},\gamma_0)=0 \quad \text{for $t=0,1,\dots,q-1$ }
\]
for some $(x_0,\gamma_0) \in
  \itvcos{0}{1}^2$, then $\mathcal{G}(g,\alpha,\beta)$ is \emph{not} a frame.
  \end{lemma}

Note that for integer oversampled Gabor systems, i.e.,
$\alpha\beta=1/q$, a zero row of the Zibulski-Zeevi matrix is the only possible reason for a failure of the
frame property. Indeed, an  integer oversampled Gabor system is a frame with bounds $A$ and
$B$ if and only if
\begin{align*}
  A \le \left(\sum_{t =0}^{q-1}
    \absbig{Z_{\tfrac{1}{\beta}}g(x+t/q,\gamma)}^2\right)^{1/2} \le B
  \quad \text{for a.e. $x,\gamma \in \itvco{0}{1}^2$.}
\end{align*}

In order to apply Lemma~\ref{lem:zero-row-failure} we need
to find $q$ simultaneous zeros of the Zak transform
along horizontal lines on $\itvco{0}{1}^2$ each uniformly separated by $1/q$.
Zeros of the Zak transform will be the theme of the next section.

\section[Zeros and symmetries of the Zak transform]{Zeros and symmetries of the Zak transform of Hermite functions}
\label{sec:zeros-symmetries-zak}

The main results of this section are Theorem~\ref{thm:symm_modu_even_n}
and~\ref{thm:symm_modu_odd_n} in Section~\ref{sec:symmetries-par} and the two lemmas
in Section~\ref{sec:additional-zeros-zak}.
Theorem~\ref{thm:symm_modu_even_n}
and~\ref{thm:symm_modu_odd_n} are a symmetry property of the function
$\R_{>0} \ni \lambda \mapsto Z_{s\lambda}g\bigl(\frac{x_0}{s^{2}},\gamma_0\bigr)$
for certain fixed values of $x_0$ and $\gamma_0$ with $s^2=2,3,4$ and $g$ being
a sufficiently nice eigenfunction of the Fourier transform. More
precisely, the same symmetry property will hold for \emph{one} $\gamma_0$ value,
but $s^2$ different $x_0$-values each separated by $1/s^2$. In
these cases, we show that
\begin{equation}
\R \ni \kappa \mapsto Z_{s2^\kappa}g\bigl(\frac{x_0}{s^{2}},\gamma_0\bigr)\label{eq:Zak-symmetry-kappa}
\end{equation}
is a bounded, continuous, and, more importantly, either an even or odd function.

The so-called even or odd modular characteristics of the function \eqref{eq:Zak-symmetry-kappa} shed new light on zeros of the Zak transform for Hermite functions. Evidently, they show that the zeros of \eqref{eq:Zak-symmetry-kappa} are symmetric around $\kappa=0$. Moreover, we will see that all ``non-trivial'' zeros discovered in \cite{WOS:A1983QC19800014} are associated with \emph{odd} modular characteristics and arise from the simple fact that an \emph{odd} and continuous
function on the real line has a zero at the origin, see Corollary~\ref{lem:extension-identities} and Remark~\ref{rem:four-zeros}. 
The known counterexamples of the frame set conjecture from \cite{Lemvigsome2017} precisely correspond to the cases where \eqref{eq:Zak-symmetry-kappa} is an odd
function (which follows from Theorem~\ref{thm:symm_modu_even_n}
and~\ref{thm:symm_modu_odd_n} with $\ell=2,3$). Note that Theorem~\ref{thm:symm_modu_odd_n} is not needed for the counterexamples presented in Section~\ref{sec:counterexamples}, but it is included because it sheds new light on the structure of the Zak transform for Hermite functions.

For the focus of this work, the remaining open cases for the Hermite frame set
conjecture correspond to \emph{even} modular characteristics of
\eqref{eq:Zak-symmetry-kappa} which do not guarantee the zeros of the Zak
transform needed in Corollary~\ref{cor:obstr-frame-prop-even-and-odd} since
\emph{even} and continuous functions on the real line do not necessarily have
any zeros. Hence, we have to work harder for the cases $h_n$, $n=4m$ and $n=4m+1$,
$m \in \Z_{>0}$. This is done in Section~\ref{sec:additional-zeros-zak}, where
we prove the existence of positive zeros of the function
\eqref{eq:Zak-symmetry-kappa} for $g=h_{n}$, $n > 2$.
The symmetry property of \eqref{eq:Zak-symmetry-kappa} will allow us to increase the
number of zeros and, hence, increase the number of counterexamples. We postpone this
to Section~\ref{sec:counterexamples}, where we also show how the symmetry property can be used to
improve the stability of numerical investigations of the frame set conjecture.

\subsection{Symmetries with respect to the time and frequency variables}
\label{sec:symmetries-va}

The Zak transform inherits symmetries of the function $f$ it acts on. In this subsection,
we recall some basic results of the Zak transform of real, imaginary, even and odd functions.
Recall that the Hermite functions are real-valued and either even or odd functions.

\begin{lemma} \label{lem:ZakProperties}
Let $f\in W(\R)$ be a continuous function. Let $m\in \Z$, $\lambda >0$, and let $x,\gamma \in \R$.
    \begin{enumerate}[(i)]
        \item Suppose $f$ is either an even or odd function. Let $j$ be $0$ if $f$ is even and $1$ if $f$ is odd. Then
        \begin{align} \label{eq:Zak_even_and_odd}
			Z_\lambda f(x,\gamma) &=(-1)^j Z_\lambda f(-x,-\gamma),
		\end{align}
		in particular
		\begin{align}
			Z_\lambda f\left(x+\tfrac{1}{2},\tfrac{m}{2}\right) &= (-1)^{j+m} Z_\lambda f\left(-x+\tfrac{1}{2},\tfrac{m}{2}\right), \quad \text{and} \label{eq:Zak_even_2}\\
			Z_\lambda f\left(x,\tfrac{m}{2}\right) &= (-1)^j Z_\lambda f\left(-x,\tfrac{m}{2}\right). \label{eq:Zak_even_and_odd2}
		\end{align}
		\label{item:Zak-symm-even-odd}
    \item Suppose $f$ is either a real or imaginary function. Let $k$ be 0 if $f$ is real and 1 if $f$ is imaginary. Then
    \begin{align}
        Z_\lambda f(x,\gamma)=(-1)^{k}\overline{Z_\lambda f(x,-\gamma)}.
        \label{eq:Zak-re-im-symm}
    \end{align}
    		\label{item:Zak-symm-real-imag}
    \end{enumerate}
\end{lemma}
\begin{proof}
\eqref{item:Zak-symm-even-odd}: Let $x,\gamma \in \R$. By definition of $j$, we have
$f(x)=(-1)^jf(-x)$. Equation~\eqref{eq:Zak_even_and_odd} is readily verified:
\begin{align*}
    Z_\lambda f(x,\gamma)&=\sqrt{\lambda}\sum_{k\in \Z}f(\lambda(x+k))e^{-2\pi i \gamma k}\\
    &=\sqrt{\lambda}\sum_{k^{\prime}\in \Z}f(\lambda(x-k^{\prime}))e^{2 \pi i \gamma k^{\prime}}\\
    &=\sqrt{\lambda}\sum_{k^{\prime}\in \Z}f(-\lambda(-x+k^{\prime}))e^{2\pi i\gamma k^{\prime}}\\
    &=(-1)^{j}Z_\lambda f(-x,-\gamma).
\end{align*}
The particular cases easily follows by
equation~\eqref{eq:Zak_even_and_odd} and the quasi-periodicity of the
Zak transform.

\eqref{item:Zak-symm-real-imag}: By definition,
$f(x)=(-1)^k\overline{f(x)}$. Equation~\eqref{eq:Zak-re-im-symm} now follows by a
similar computation as in
the proof of part~\eqref{item:Zak-symm-even-odd}.
\end{proof}

Suppose $f$ is either even or odd and takes either real or imaginary values. With the definitions of $j$ and $k$ as in Lemma~\ref{lem:ZakProperties} above, we then have, by combining~\eqref{eq:Zak_even_and_odd} and~\eqref{eq:Zak-re-im-symm},
\begin{align} 
\label{eq:Zak-symm-evenodd-real-imag}
			Z_\lambda f(x,\gamma) =(-1)^j Z_\lambda f(-x,-\gamma) =(-1)^k \overline{Z_\lambda f(x,-\gamma)} \quad \text{for all }x,\gamma \in \R.
		\end{align}
The second equality in~\eqref{eq:Zak-symm-evenodd-real-imag} tells us that the Zak transform is also reflection (anti-)symmetric with respect to the first variable, i.e.,  $Z_\lambda f(x,\gamma) =(-1)^{k+j} \overline{Z_\lambda f(-x,\gamma)}$, and that 
        \begin{align*} 
        	\vert Z_\lambda f(x,\gamma)\vert = \vert Z_\lambda f(\pm x,\pm \gamma)\vert \qquad \text{for
          all }x,\gamma \in \R,
        \end{align*}
cf. equation (24) in \cite{WOS:A1983QC19800014}.

\begin{corollary}
\label{cor:Zak_even_odd_zeroes}
    Let $f\in W(\R)$ be continuous and $\lambda >0$.
    \begin{enumerate}[(i)]
        \item If $f$ is even, then
        \begin{align}
            Z_\lambda f(x_0,\gamma_0) &=0, \text{ for } (x_0,\gamma_0)\in \mathbb{Z}^2 + \left(\tfrac{1}{2},\tfrac{1}{2}\right) \label{eq:Zak_even_3}.
        \end{align}
        \item If $f$ is odd, then
        \begin{align*}
            Z_\lambda f(x_0,\gamma_0) &=0, \text{ for } (x_0,\gamma_0)\in					\frac{1}{2}\mathbb{Z}^2\backslash \left( \mathbb{Z}^2 + \left(\tfrac{1}{2},\tfrac{1}{2}\right)\right).
        \end{align*}
        \label{item:Zak-zeros-odd}
    \end{enumerate}
\end{corollary}
\begin{proof}
    First, let $f\in W(\R)$ be continuous and even. Then equation~\eqref{eq:Zak_even_2} with $m=1$ reads $Z_\lambda f(x+\tfrac{1}{2},\tfrac{1}{2})=-Z_\lambda f(-x+\tfrac{1}{2},\tfrac{1}{2})$. Since $Z_\lambda f(x+\tfrac{1}{2},\tfrac{1}{2})$ is an odd function with respect to the $x$-variable, it follows, by taking $x=0$, that $Z_\lambda f(\tfrac{1}{2},\tfrac{1}{2})=0$. The other zeros of \eqref{eq:Zak_even_3} now follow from quasi-periodicity of the Zak transform. 
    
    The zeros in \eqref{item:Zak-zeros-odd} can be shown in a similar manner. For $f\in W(\R)$ continuous and odd, it follows from equation~\eqref{eq:Zak_even_and_odd2} that $Z_\lambda f(x,\tfrac{m}{2})$ ($m=0,1$) is an odd function with respect to $x$. Setting $x=0$ yields $Z_\lambda(0,\tfrac{m}{2})=0$ for $m=0,1$.
    Finally, with $m=0$ in equation~\eqref{eq:Zak_even_2}, we see that 
    $Z_\lambda f(x+1/2,0)$ is an odd function in $x$. Thus $Z_\lambda f$ has a zero at $(1/2,0)$. As above, the remaining zeros follow from quasi-periodicity.
\end{proof}

By the quasi-periodicity~\eqref{eq:Zak-quasi-periodicity}, the value of $Z_\lambda f$ of general functions $f$ in $L^2(\R)$ is completely determined by its value on the unit square $\itvco{0}{1}^2$ or, more generally, on any measurable set $S$ that (up to set of measure zero) tiles $\R^2$ by $\Z^2$-translations. The set $S$
is said to be a \emph{fundamental domain} of the Zak transform of functions in  $L^2(\R)$.

For even and odd functions, it follows by~\eqref{eq:Zak_even_and_odd} and~\eqref{eq:Zak_even_and_odd2}, respectively,  that the fundamental domain is $\itvco{0}{1/2}\times \itvco{0}{1}$ or $\itvco{0}{1}\times \itvco{0}{1/2}$ or, in general, any set of measure $1/2$ that tiles $\R^2$ by translations by $\Z^2$ and reflections with respect to the origo. If an even and odd function in $L^2(\R)$ is, in addition, real-valued or imaginary-valued, the fundamental domain again shrinks by a factor two, e.g., $\itvco{0}{1/2}^2$ is a fundamental domain and in general any set of measure $1/4$ that tiles $\R^2$ by translations by $\Z^2$ and reflections with respect to the origo \emph{and} either of the axes.

\subsection{Obstructions to the frame property}
\label{sec:obstr-frame-prop}

For odd and continuous functions $g\in W(\R)$ Lyubarskii and
Nes~\cite{MR3027914} showed that $\mathcal{G}(g,\alpha, \beta)$ fails to be a
frame along any of the hyperbolas $\alpha\beta= \tfrac{q-1}{q}$ for $q \in \Z_{>0}$.
Combining Lemma~\ref{lem:ZakProperties} and Lemma~\ref{lem:zero-row-failure} with Corollary \ref{cor:Zak_even_odd_zeroes}, we have the following point failure
 of the frame property for even and odd functions, which serves as a basis for all our counterexamples.

\begin{corollary}
  \label{cor:obstr-frame-prop-even-and-odd}
  \begin{enumerate}[(a)]
    \item Let  $\lambda>0$ and let $g \in W(\R)$ be an even, continuous
          function. \label{item:counterexamples-a-even} \label{item:cor-obsr-even}
          \begin{enumerate}[(i)]
            \item If $Z_{\lambda} g(1/4,1/2)=0$, then
                  $\mathcal{G}(g,\lambda/2,1/\lambda)$ is not a frame.
                  \label{item:cx-14}
            \item If $Z_{\lambda} g(1/6,1/2)=0$, then
                  $\mathcal{G}(g,\lambda/3,1/\lambda)$ and
                  $\mathcal{G}(g,2\lambda/3,1/\lambda)$ are not frames.
                  \label{item:cx-16}
          \end{enumerate}
    \item Let  $\lambda>0$ and let $g \in W(\R)$ be an odd, continuous
          function. \label{item:cor-obsr-odd}
          \begin{enumerate}[(i)]
            \item If either $Z_{\lambda} g(1/6,0)=0$ or $Z_{\lambda} g(1/3,0)=0$, then
                  $\mathcal{G}(g,\lambda/3,1/\lambda)$ is not a frame.
            \item If $Z_{\lambda} g(1/4,0)=0$, then
                  $\mathcal{G}(g,\lambda/4,1/\lambda)$ is not a frame.
          \end{enumerate}
        \end{enumerate}
\end{corollary}
\begin{proof}
  We only proof part (ii) of \eqref{item:counterexamples-a-even} as the other
  proofs are similar. Since $g$ is even, it follows that
  $Z_{\lambda}g(1/2,1/2)=0$ for all $\lambda>0$. By \eqref{eq:Zak_even_2}, the
  assumption $Z_{\lambda} g(1/6,1/2)=0$ implies that
  $Z_{\lambda} g(5/6,1/2)=0$. Thus, the Zak transform of
  $Z_{\lambda}g(x,\gamma)$ has zeros $(k/3+1/6,1/2), k \in \Z$, along the
  horizontal line $\gamma=1/2$, each separated by multiples of $1/3$, so we
  can apply Lemma~\ref{lem:zero-row-failure} with $x_{0} = 1/6$ and
  $\gamma_{0}=1/2$, $q=3$ and $\frac{1}{\beta}=\lambda$. In both cases
  $\alpha\beta=1/3$ (i.e., $\alpha=\lambda/3$) and $\alpha\beta=2/3$ (i.e.,
  $\alpha=2\lambda/3$), we conclude that $\mathcal{G}(g,\alpha,\beta)$ is
  \emph{not} a frame.
\end{proof}


\subsection{Symmetries with respect to the modular parameter}
\label{sec:symmetries-par}

\begin{lemma} \label{lem:ZakPoisson}
Let $g\in E_\ell\cap W(\R)$ for $\ell=0,1,2,3$ and $\lambda>0$.
For $s^2 \in \Z_{> 0}$ and $x\in \R$, it holds:
\begin{align} \label{eq:symmetries}
Z_{s\lambda} g\bigl(\tfrac{x+p}{s^2},x\bigr) = (-i)^\ell 
  \myexp{2 \pi i x \tfrac{x+p}{s^2}} \frac{1}{s} \sum_{\rpar=0}^{s^2-1} \myexp{2\pi i \rpar \tfrac{x+p}{s^2}} Z_{s/\lambda} g\bigl(\tfrac{x+\rpar}{s^2},-x\bigr)
\end{align}
for all $p=0,\dots, s^2-1$.
\end{lemma}

\begin{proof}
An application of Poisson's summation formula~\eqref{eq:F-of-Zak} to the left-hand side of equation~\eqref{eq:symmetries} yields
\begin{align}
    Z_{s\lambda} g\bigl(\tfrac{x+p}{s^2},x\bigr) &= \exp\Bigl(2\pi i \frac{x+p}{s^2}x\Bigr) Z_{\frac{1}{s\lambda}}\hat{g}\bigl(x,-\tfrac{x+p}{s^2}\bigr) \nonumber \\
    \quad &= \exp\Bigl(2\pi i \frac{x+p}{s^2}x\Bigr)(-i)^\ell\frac{1}{\sqrt{s\lambda}}\sum_{k\in \Z} g\bigl(\tfrac{1}{s\lambda}(x+k)\bigr)\exp\Bigl(2\pi i \frac{x+p}{s^2}k\Bigr) \label{eq:symmetriesProof1},
\end{align}
where we have used that $g\in E_{\ell}$ is an eigenfunction of the Fourier transform in the final step.
For brevity, let $\eta = \exp\left(2\pi i x\frac{x+p}{s^2}\right)(-i)^\ell$, and write $\left(\frac{1}{s\lambda}(x+k)\right)$ as $\left(\frac{s}{\lambda} \cdot \frac{x+k}{s^2}\right)$. The series in equation~\eqref{eq:symmetriesProof1} can be split into $s^2$ series by the change of variables $k = \rpar + s^2m$, where $\rpar\in\{0,\ldots,s^2-1\}, m\in\Z$:
\begin{align*}
    Z_{s\lambda} g\bigl(\tfrac{x+p}{s^2},x\bigr) &= \frac{\eta}{\sqrt{s\lambda}} \sum_{\rpar=0}^{s^2-1}\sum_{m\in \Z} g\Bigl(\frac{s}{\lambda} \cdot \frac{x+\rpar+s^2m}{s^2}\Bigr)\exp\Bigl(2\pi i \frac{x+p}{s^2}(\rpar+s^2m)\Bigr)\\
    &=  \frac{\eta}{\sqrt{s\lambda}} \sum_{\rpar=0}^{s^2-1} \exp\Bigl(2\pi i \rpar \frac{x+p}{s^2}\Bigr)  \sum_{m\in \Z} g\Bigl(\frac{s}{\lambda} \cdot \frac{x+\rpar}{s^2}+m\Bigr)\exp\Bigl(2\pi i xm\Bigr).
\end{align*}
For a fixed $\rpar$ the series over $m\in \Z$ can be identified as a Zak transformation, up to a
missing scaling factor $\sqrt{s/\lambda}$. Hence, we arrive at:
\begin{align*}
    Z_{s\lambda} g\bigl(\tfrac{x+p}{s^2},x\bigr) &= \eta  \frac{1}{s}\sum_{\rpar=0}^{s^2-1} \exp\Bigl(2\pi i \rpar \frac{x+p}{s^2}\Bigr) Z_{\frac{s}{\lambda}}g\Bigl(\frac{x+\rpar}{s^2},-x\Bigr).
\end{align*}
Inserting the value of $\eta$ yields the desired identity.
\end{proof}

\begin{lemma} \label{lem:ZakPoissonCases}
Let $g \in W(\R)$ be an even or odd function   and $\lambda>0$. Define $j$ to be $0$
if $g$ is even and $1$ if $g$ is odd. Let $s^2 \in \Z_{> 0}$ and $p=0,\dots, s^2-1$.
\begin{enumerate}[(i)]
    \item For $s^2$ even, we have:
\begin{align}
\sum_{\rpar=0}^{s^2-1} \myexp{2\pi i \rpar \tfrac{\tfrac{1}{2}+p}{s^2}} Z_{s/\lambda} g\bigl(\tfrac{\tfrac{1}{2}+\rpar}{s^2},-\tfrac{1}{2}\bigr)
&= \sum_{\rpar=0}^{\frac{s^{2}}{2}-1} \bigl(\myexp{2\pi i \rpar
                                                                                                                                         \tfrac{\tfrac{1}{2}+p}{s^2}}+(-1)^{j}\myexp{2\pi i((s^{2}-1-\rpar)\tfrac{\tfrac{1}{2}+p}{s^{2}}+\tfrac{1}{2})}\bigr) \nonumber \\ &\cdot Z_{s/\lambda} g\bigl(\tfrac{\tfrac{1}{2}+\rpar}{s^2},-\tfrac{1}{2}\bigr); \label{eq:s2-even-half}
\\
\sum_{\rpar=0}^{s^2-1} \myexp{2\pi i \rpar \tfrac{p}{s^2}} Z_{s/\lambda} g\bigl(\tfrac{\rpar}{s^2},0\bigr)
&= \sum_{\rpar=1}^{\frac{s^{2}}{2}-1} \bigl(\myexp{2\pi i \rpar
                                                                                                      \tfrac{p}{s^2}}+(-1)^{j}\myexp{2\pi i \frac{p}{s^{2}}(s^{2}-\rpar)} \bigr) \cdot Z_{s/\lambda} g\bigl(\tfrac{\rpar}{s^2},0\bigr)    \nonumber \\&+Z_{s/\lambda} g\bigl(0,0\bigr)+\myexp{2\pi i \frac{p}{2}}Z_{s/\lambda} g\bigl(\tfrac{1}{2},0\bigr); \label{eq:s2-even-zero}
    \end{align} \label{item:s2-even}
\item For $s^2$ odd, we have:
\begin{align}
\sum_{\rpar=0}^{s^2-1} \myexp{2\pi i \rpar \tfrac{\tfrac{1}{2}+p}{s^2}} Z_{s/\lambda} g\bigl(\tfrac{\tfrac{1}{2}+\rpar}{s^2},-\tfrac{1}{2}\bigr)
&= \sum_{\rpar=0}^{\frac{s^{2}-1}{2}-1}\bigl(\myexp{2\pi i\rpar\frac{\tfrac{1}{2}+p}{s^{2}}}+(-1)^{j}\myexp{2\pi i((s^{2}-1-\rpar)\frac{\tfrac{1}{2}+p}{s^{2}}+\tfrac{1}{2})}\bigr) \nonumber \\ &\cdot Z_{s/\lambda} g\bigl(\tfrac{\tfrac{1}{2}+\rpar}{s^2},-\tfrac{1}{2}\bigr)+ \myexp{2\pi i \frac{s^{2}-1}{2}\frac{\tfrac{1}{2}+p}{s^{2}}}Z_{s/\lambda} g\bigl(\tfrac{1}{2},-\tfrac{1}{2}\bigr); \label{eq:s2-odd-half}\\
\sum_{\rpar=0}^{s^2-1} \myexp{2\pi i \rpar \tfrac{p}{s^2}} Z_{s/\lambda}
  g\bigl(\tfrac{\rpar}{s^2},0\bigr)
  &=\sum_{\rpar=1}^{\frac{s^{2}-1}{2}}\bigl(\myexp{2\pi i\rpar
    \frac{p}{s^{2}}}+(-1)^{j}\myexp{2\pi i
    \frac{p}{s^{2}}(s^{2}-\rpar)}\bigr) \cdot Z_{s/\lambda}
    g\bigl(\tfrac{\rpar}{s^2},0\bigr) \nonumber \\ &+Z_{s/\lambda}
                                                    g\bigl(0,0\bigr); \label{eq:s2-odd-zero}
\end{align} \label{item:s2-odd}
\end{enumerate}
\end{lemma}

\begin{proof}
We first consider the sum
 \begin{equation}
    \sum_{\rpar=0}^{s^2-1} \myexp{2\pi i \rpar \tfrac{\tfrac{1}{2}+p}{s^2}}
    Z_{s/\lambda}
    g\bigl(\tfrac{\tfrac{1}{2}+\rpar}{s^2},-\tfrac{1}{2}\bigr).
  \label{eq:at-one-half}
\end{equation}
By the symmetry property~\eqref{eq:Zak_even_2}, we see that
  \[
    Z_{s/\lambda}g\bigl(-\frac{\frac{1}{2}+k}{s^{2}}+1,-\frac{1}{2}\bigr)=(-1)^{j}\myexp{2\pi
      i
      \frac{1}{2}}Z_{s/\lambda}g\bigl(\frac{\frac{1}{2}+k}{s^{2}},-\frac{1}{2}\bigr)
    \quad k\in \bigg\lbrace 0,1,2,\cdots,\Big\lceil \frac{s^{2}}{2}-1\Big\rceil \bigg\rbrace.
    \]
    which identifies a pairing of the terms in the
    sum~\eqref{eq:at-one-half} by combining and factorizing the summands $\rpar=0$ and
    $\rpar=s^{2}-1$, $\rpar=1$ and $\rpar=s^{2}-2$ and so forth. If
    $s^{2}$ is even, we can pair all terms in the sum this way, and if
    $s^2$ is odd we leave the central term $\rpar=\tfrac{s^2-1}{2}$
    unpaired. This pairing then yields the relations
    \eqref{eq:s2-even-half} and \eqref{eq:s2-odd-half}.

Now we consider the sum $\sum_{\rpar=0}^{s^2-1} \myexp{2\pi i \rpar
  \tfrac{p}{s^2}} Z_{s/\lambda} g\bigl(\tfrac{\rpar}{s^2},0\bigr)$.
Again, by symmetry~\eqref{eq:Zak_even_2}, we see that 
\begin{equation}
    Z_{s/\lambda}g\bigl(-\tfrac{k}{s^{2}}+1,0\bigr)=(-1)^{j}Z_{s/\lambda}g\bigl(\tfrac{k}{s^{2}},0\bigr),\quad k\in\bigg\lbrace 1,2,\cdots,\Big\lceil \frac{s^{2}}{2}-1\Big\rceil \bigg\rbrace.
    \label{eq:pairing2}
\end{equation}
Equation (\ref{eq:pairing2}) suggests a pairing of the Zak transforms
in $\sum_{\rpar=0}^{s^2-1} \myexp{2\pi i \rpar \tfrac{p}{s^2}} Z_{s/\lambda}
g\bigl(\tfrac{\rpar}{s^2},0\bigr)$, by pairing $\rpar=1$ with
$\rpar=s^{2}-1$, $\rpar=2$ with $\rpar=s^{2}-2$ and so forth. This leaves
the term associated with $\rpar=0$ not being paired. Hence, if
    $s^{2}$ is odd we can pair all terms, excluding $\rpar=0$, in the sum this way, and if
    $s^2$ is even, we leave the term $\rpar=\tfrac{s^2-1}{2}$
    (as well as $\rpar=0$)
    unpaired. This pairing immediately yields the relations
    \eqref{eq:s2-even-zero} and \eqref{eq:s2-odd-zero}.
\end{proof}

For \emph{even} Hermite functions, or more generally, for functions in
$E_\ell\cap W(\R)$, $\ell=0,2$, we have the following symmetry property as
illustrated in Figure~\ref{fig:thm_symmetry_even}.
\begin{figure}
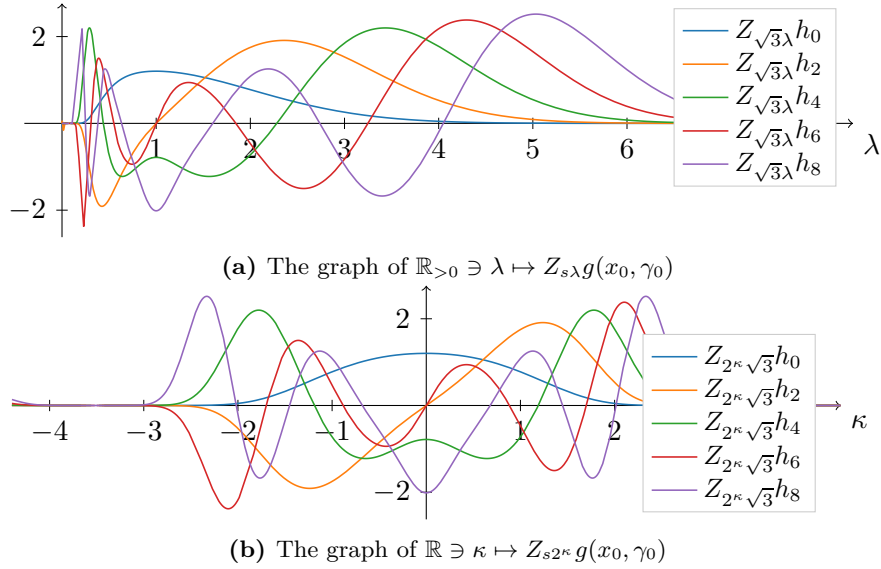

     \centering
     \begin{subfigure}[b]{1\textwidth}
         \centering
  \centering

  \caption{The graph of $\R_{>0} \ni  \lambda \mapsto Z_{s\lambda}g(x_{0},\gamma_{0})$}
         \label{fig:zak_even_pos_lambda}
     \end{subfigure}
     \begin{subfigure}[b]{1\textwidth}
  \centering
%

         \caption{The graph of $\R \ni \kappa \mapsto Z_{s2^{\kappa}}g(x_{0},\gamma_{0})$}
         \label{fig:zak_even_real_kappa}
     \end{subfigure}
\caption{The Zak transform at $(x_{0},\gamma_{0})$ as a function of the modular
  parameter for $g=h_n$, $n=0,2,4,\dots, 8$ with
  $(x_{0},\gamma_{0})=\bigl(\frac{1/2+p}{s^{2}},\tfrac{1}{2}\bigr) = (1/6,1/2)$,
  where $s^{2}=3$ and $p=0$. The symmetry property from
  Theorem~\ref{thm:symm_modu_even_n} is most apparent from
  Figure~\ref{fig:zak_even_real_kappa}. The Zak transform
  $\R \ni \kappa \mapsto Z_{s2^{\kappa}}g(x_{0},\gamma_{0})$ is even for
  $g=h_{4m}$ and odd for $g=h_{4m-2}$ for $m\in \Z_{\ge 0}$. Note the truncation errors
  for $\kappa \le -4$; we comment on this issue in Section~\ref{sec:numer-exper}.}
\label{fig:thm_symmetry_even}
   \end{figure}

\begin{theorem}
  \label{thm:symm_modu_even_n}
For $g\in E_\ell\cap W(\R)$, $\ell=0,2$, and $s^{2}\in \{2,3\}$, it holds for any $\lambda>0$:
\begin{equation}
    Z_{s\lambda}g\bigl(\frac{\tfrac{1}{2}+p}{s^{2}},\tfrac{1}{2}\bigr)
    =
    (-1)^{\floor{\ell/2}}Z_{s/\lambda}g\bigl(\frac{\tfrac{1}{2}+p}{s^{2}},\tfrac{1}{2}\bigr) \quad \text{for
      all } p=0,1,\ldots,s^{2}-1.
  \label{eq:Zak-symmetry-one-half}
  \end{equation}
  In case $s^2=2$, equation~\eqref{eq:Zak-symmetry-one-half} also
  holds for $\ell=1,3$. 
\end{theorem}
\begin{proof}
  We consider first the case $s^{2}=2$. Hence, let $g\in E_\ell\cap
  W(\R)$ for $\ell=0,1,2,3$. In this case Lemma~\ref{lem:ZakPoisson}
  yields 
\begin{align}
  Z_{\sqrt{2}\lambda} g\bigl(\frac{\tfrac{1}{2}+p}{2},\tfrac{1}{2}\bigr) & =
\myexp{-2\pi i \tfrac{\ell}{4}} \myexp{2\pi i \frac{\tfrac{1}{2}+p}{4}}\tfrac{1}{\sqrt{2}} \sum_{\rpar=0}^{1}\myexp{2\pi i\rpar   \frac{\tfrac{1}{2}+p}{2}}Z_{\sqrt{2}/\lambda}g\bigl(\tfrac{\tfrac{1}{2}+\rpar}{2},-\tfrac{1}{2}\bigr).
\label{eq:zaksqrt2Estimate1}
\end{align}
 By Lemma \ref{lem:ZakProperties}(i) and
 equation~\eqref{eq:Zak-quasi-periodicity}, the sum in
 \eqref{eq:zaksqrt2Estimate1} can be rewritten as
\begin{align*}
   \sum_{\rpar=0}^{1}\myexp{2\pi i\rpar \frac{\tfrac{1}{2}+p}{2}}Z_{\sqrt{2}/\lambda}g\bigl(\tfrac{\tfrac{1}{2}+\rpar}{2},-\tfrac{1}{2}\bigr) &= \bigl(\myexp{2\pi i p \frac{\tfrac{1}{2}+p}{2}}+\myexp{2\pi i[(1-p)\frac{\tfrac{1}{2}+p}{2}-\tfrac{j+1}{2}]}\bigr) Z_{\sqrt{2}/\lambda}g\bigl(\tfrac{\tfrac{1}{2}+p}{2},-\tfrac{1}{2}\bigr) \\ 
   &= \bigl(\myexp{2\pi i p \frac{\tfrac{1}{2}+p}{2}}+\myexp{2\pi i[(1-p)\frac{\tfrac{1}{2}+p}{2}-\tfrac{j+1}{2}]}\bigr) Z_{\sqrt{2}/\lambda}g\bigl(\tfrac{\tfrac{1}{2}+p}{2},\tfrac{1}{2}\bigr), 
\end{align*}
where $j=0$ if $\ell =0,2$ and $j=1$ if $\ell=1,3$. Hence, to show
\eqref{eq:Zak-symmetry-one-half} for $s^2=2$, we have only left to show that the
phase factor
\begin{align*}
   \myexp{-2\pi i\tfrac{\ell}{4}}\myexp{2\pi i\frac{\tfrac{1}{2}+p}{4}}\tfrac{1}{\sqrt{2}}
    \bigl(\myexp{2\pi i p \frac{\tfrac{1}{2}+p}{2}}+\myexp{2\pi
  i[(1-p)\frac{\tfrac{1}{2}+p}{2}-\tfrac{j+1}{2}]}\bigr)
\end{align*}
equals $(-1)^{\ceil{\ell/2}}$. We first compute
\begin{align*}
  \myexp{2\pi i\frac{\tfrac{1}{2}+p}{4}} 
    \bigl(\myexp{2\pi i p \frac{\tfrac{1}{2}+p}{2}}+\myexp{2\pi i[(1-p)\frac{\tfrac{1}{2}+p}{2}-\tfrac{j+1}{2}]}\bigr)&=\myexp{2\pi i(p+\tfrac{1}{2})\frac{\tfrac{1}{2}+p}{2}}+(-1)^{j+1}\myexp{2\pi i (\tfrac{3}{2}-p)\frac{\tfrac{1}{2}+p}{2}} \\
    &=\myexp{\pi i
      p(1+p)}\myexp{\frac{\pi i}{4}} +(-1)^{\ell+1} \myexp{\pi ip(1-p)}\myexp{\frac{3\pi i }{4}}\\
    &=\myexp{\frac{\pi
      i}{4}}+(-1)^\ell \myexp{-\frac{\pi i}{4}}
\end{align*}
where the third equality holds since $p(1\pm p)$ is even for any $p\in \Z$.
Therefore,
\begin{align*}
   \myexp{-2\pi i\tfrac{\ell}{4}}\myexp{2\pi i\frac{\tfrac{1}{2}+p}{4}}\tfrac{1}{\sqrt{2}}
    \bigl(\myexp{2\pi i p \frac{\tfrac{1}{2}+p}{2}}+\myexp{2\pi i[(1-p)\frac{\tfrac{1}{2}+p}{2}-\tfrac{j+1}{2}]}\bigr)&=\frac{1}{\sqrt{2}}\myexp{-2\pi i\tfrac{\ell}{4}} \bigl(\myexp{\frac{\pi
      i}{4}}+(-1)^\ell \myexp{-\frac{\pi i}{4}}\bigr)\\
    &=
      \begin{cases}
        1 & \ell=0,1, \\
        -1 & \ell=2,3,
      \end{cases}
\end{align*}
which is what we had to show. 

We now turn to the case $s^{2}=3$. Let $g \in E_\ell \cap W(\R)$ for $\ell=0$ or
$\ell=2$. For $p=1$ equation~\eqref{eq:Zak-symmetry-one-half} trivially
holds as $Z_\lambda g(\tfrac{1}{2},\tfrac{1}{2})=0$ for any
$\lambda>0$ whenever $g$ is even, see~\eqref{eq:Zak_even_3}.
Hence, we only have to consider $p\in \{0,2\}$. Using Lemma~\ref{lem:ZakPoisson} we get
\begin{equation}
    Z_{\sqrt{3}\lambda}g\bigl(\tfrac{\tfrac{1}{2}+p}{3},\tfrac{1}{2}\bigr)=(-1)^{\ell/2}\myexp{2\pi i\frac{\tfrac{1}{2}+p}{6}}\frac{1}{\sqrt{3}}\sum_{\rpar=0}^{2}\myexp{2\pi i \rpar \frac{\tfrac{1}{2}+p}{3}}Z_{\sqrt{3}/\lambda}g\bigl(\tfrac{\tfrac{1}{2}+\rpar}{3},-\tfrac{1}{2}\bigr).
    \label{eq:zaksqrt3EST}
  \end{equation}
As above, we first rewrite the sum in \eqref{eq:zaksqrt3EST}:
\begin{align*}
    \sum_{\rpar=0}^{2}\myexp{2\pi i \rpar
\frac{\tfrac{1}{2}+p}{3}}Z_{\sqrt{3}/\lambda}g\bigl(\tfrac{\tfrac{1}{2}+\rpar}{3},-\tfrac{1}{2}\bigr) = \bigl(
  \myexp{2\pi i p \frac{\tfrac{1}{2}+p}{3}}+\myexp{2\pi i
  [(2-p)\frac{\tfrac{1}{2}+p}{3}-\tfrac{1}{2}]}\bigr) Z_{\sqrt{3}/\lambda}g\bigl(\tfrac{\tfrac{1}{2}+p}{3},\tfrac{1}{2}\bigr)
\end{align*}
using that  $Z_\lambda g(\tfrac{1}{2},-\tfrac{1}{2})=0$ for any
$\lambda>0$ since $g$ is even. We then compute the phase factor: 
\begin{multline*}
  (-1)^{\ell/2} \myexp{2\pi
  i\frac{\tfrac{1}{2}+p}{6}}\tfrac{1}{\sqrt{3}}\bigl( \myexp{2\pi i p
  \frac{\tfrac{1}{2}+p}{3}}+\myexp{2\pi i
  [(2-p)\frac{\tfrac{1}{2}+p}{3}-\tfrac{1}{2}]}\bigr) \\  = \tfrac{1}{\sqrt{3}} (-1)^{\ell/2}
                                                         \bigl(\myexp{2\pi
                                                         i
                                                         (p+\tfrac{1}{2})\frac{\tfrac{1}{2}+p}{3}}
                                                         -  \myexp{2\pi i (\tfrac{5}{2}-p)\frac{\tfrac{1}{2}+p}{3}}\bigr)\\
    =\tfrac{1}{\sqrt{3}} (-1)^{\ell/2} \bigl( \myexp{\frac{\pi i}{6}}\myexp{2\pi i
      \tfrac{p(1+p)}{3}} - \myexp{\frac{5\pi i}{6}}\myexp{2\pi i\tfrac{p(2-p)}{3}} \bigr)\\
    =\tfrac{1}{\sqrt{3}}(-1)^{\ell/2}\bigl(\myexp{\frac{\pi i}{6}}-\myexp{\frac{5\pi i }{6}} \bigr)
    =(-1)^{\ell/2},
  \end{multline*}
where in the third equality we used that $p \neq 1 (\mathop{mod} 3)$. Combining
the above three displayed equations yields the desired identity.  
\end{proof}

For \emph{odd} Hermite functions, or more generally, for functions in
$E_\ell\cap W(\R)$, $\ell=1,3$, the symmetry properties in
Theorem~\ref{thm:symm_modu_even_n} does not hold, see
Figure~\ref{fig:thm_no_symmetry_odd}. However, it is possible to find another
similar symmetry property for these odd functions as detailed in Theorem~\ref{thm:symm_modu_odd_n} below. 
\begin{figure}
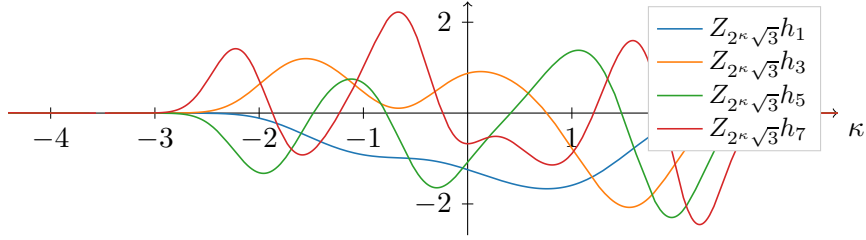

  \centering


         \caption{For the same parameter choices as in
           Figure~\Ref{fig:thm_symmetry_even} the Zak transform
  $\kappa \mapsto Z_{s2^{\kappa}}g(x_{0},\gamma_{0})$ is neither even nor
  odd for $g=h_{n}$ when $n$ is \emph{odd}. The figure illustrates the cases $n=1,3,5,7$.}
\label{fig:thm_no_symmetry_odd}
\end{figure}

\begin{theorem}
    \label{thm:symm_modu_odd_n}
For $g\in E_\ell\cap W(\R)$, $\ell=1,3$, and $s^{2}\in \{3,4\}$, it holds for any $\lambda>0$:
\begin{equation}
    Z_{s\lambda}g\bigl(\frac{p}{s^{2}},0\bigr)
    =
    (-1)^{(\ell-1)/2}Z_{s/\lambda}g\bigl(\frac{p}{s^{2}},0\bigr) \quad \text{for
      all } p=0,1,\ldots,s^{2}-1.
  \label{eq:Zak-symmetry-zero}
  \end{equation}
\end{theorem}
\begin{proof}
  We first consider $s^{2}=3$. As $g\in E_\ell\cap W(\R)$ for $\ell=1, 3$ is an
  odd function, we have by
  Corollary~\ref{cor:Zak_even_odd_zeroes}\eqref{item:Zak-zeros-odd} that, if $p=0$,
  then $Z_{\lambda}g\bigl(\frac{p}{3},0 \bigr)=0$ for all $\lambda > 0$. So, let
  $p \in \{1,2\}$. In this case Lemma~\ref{lem:ZakPoisson} yields
  \begin{equation}
      Z_{\sqrt{3}\lambda}g\bigl(\frac{p}{3},0\bigr)=(-i)^{\ell} \frac{1}{\sqrt{3}}\sum_{r=1}^{2}e^{2\pi i r\frac{p}{3}}Z_{\sqrt{3}/\lambda}g(\frac{r}{3},0), \label{eq:odd_sum_s3}
  \end{equation}
  since $Z_{\sqrt{3}/\lambda}g(0,0)=0$. By Lemma \ref{lem:ZakProperties}(i) the
  sum in \eqref{eq:odd_sum_s3} can be factored as
  \begin{align*}
      &\sum_{r=1}^{2}e^{2\pi i r\frac{p}{3}}Z_{\sqrt{3}/\lambda}g(\frac{r}{3},0)= \bigl(e^{2\pi i \frac{p^{2}}{3}}-e^{2\pi i [(3-p)\frac{p}{3}]}\bigr) Z_{\sqrt{3}/\lambda}g(\frac{p}{3},0).
  \end{align*}
  Thus, to show the equality in \eqref{eq:Zak-symmetry-zero}, it suffices to show that the factor
  \begin{equation*}
      (-i)^{\ell}\frac{1}{\sqrt{3}}\bigl(e^{2\pi i \frac{p^{2}}{3}}-e^{2\pi i [(3-p)\frac{p}{3}]}\bigr)
  \end{equation*}
  equals $(-1)^{(\ell-1)/2}$. We first compute the complex exponential
  \begin{align*}
      e^{2\pi i \frac{p^{2}}{3}}-e^{2\pi i [(3-p)\frac{p}{3}]}&= e^{2\pi \frac{i}{3}}-e^{2\pi i [(3-p)\frac{p}{3}]}\\
      &=e^{2\pi \frac{i}{3}}-e^{2\pi i\frac{2}{3}},
  \end{align*}
  where the first equation follows from $p^{2}\equiv 1 \mod 3$, and the second equality follows from $(3-p)p\equiv 2 \mod 3$. Consequently,
  \begin{align*}
      (-i)^{\ell}\frac{1}{\sqrt{3}}\bigl(e^{2\pi \frac{i}{3}}-e^{2\pi i\frac{2}{3}}\bigr)&=(-i)^{\ell-1}=(-1)^{(\ell-1)/2}, 
  \end{align*}
  which is what we wanted to show.

We now consider the case where $s^{2}=4$. Since $g$ is an odd function, Corollary~\ref{cor:Zak_even_odd_zeroes}\eqref{item:Zak-zeros-odd} shows that $Z_{\lambda}g(\frac{p}{4},0)=0$ for $p=0,2$. Consequently, we only consider $p\in \{1,3\}$. Using Lemma~\ref{lem:ZakPoisson} we have
\begin{equation}
    Z_{2\lambda}g(\frac{p}{4},0)=(-i)^{\ell}\frac{1}{2}\bigl(e^{2\pi i \frac{p}{4}}Z_{2/\lambda}g(\frac{1}{4},0)+e^{2\pi i \frac{3p}{4}}Z_{2/\lambda}g(\frac{3}{4},0)\bigr), \label{eq:odd_sum_s4}
\end{equation}
where we used that $Z_{2/\lambda}g(0,0)=Z_{2/\lambda}g(\frac{1}{2},0)=0$. Similarly to the above computation, we rewrite the sum in \eqref{eq:odd_sum_s4}
  \begin{align*}
      e^{2\pi i \frac{p}{4}}Z_{2/\lambda}g(\frac{1}{4},0)+e^{2\pi i \frac{3p}{4}}Z_{2/\lambda}g(\frac{3}{4},0)&= \bigl(e^{2\pi i \frac{p^{2}}{4}}-e^{2\pi i [(4-p)\frac{p}{4}]}\bigr) Z_{2/\lambda}g(\frac{p}{4},0).
  \end{align*}
  We then compute the complex exponential factor
  \begin{align*}
      (-i)^{\ell}\frac{1}{2}\bigl(e^{2\pi i \frac{p^{2}}{4}}-e^{2\pi i [(4-p)\frac{p}{4}]}\bigr)&=(-i)^{\ell}\frac{1}{2}\bigl(e^{2\pi \frac{i}{4}}-e^{2\pi i \frac{3}{4}}\bigr)\\ =(-1)^{(\ell-1)/2},
  \end{align*}
  where the first equality follows by $p^{2}\equiv 1 \mod 4$ and $(4-p)p\equiv 3 \mod 4$. Combining the results for the cases $s^{2}=3$ and $s^{2}=4$ provides the identity of the theorem.
\end{proof}
As $g$ in Theorem~\ref{thm:symm_modu_odd_n} is an odd function, we have by   Corollary~\ref{cor:Zak_even_odd_zeroes}\eqref{item:Zak-zeros-odd}
that $Z_{\lambda}g\bigl(\frac{p}{2},0)=0$ for all $p \in \Z$. Thus, the relation~\eqref{eq:Zak-symmetry-zero}
is also true for $s^{2}=2$.

Let us end this section by showing how we can recover the zeros found in \cite{WOS:A1983QC19800014} and re-discovered and extended to eigenspaces in \cite[Lemma 5]{Lemvigsome2017}.
\begin{corollary}[\cite{WOS:A1983QC19800014,Lemvigsome2017}]
\label{lem:extension-identities}
\begin{enumerate}[(i)]
\item For $g \in E_2 \cap W(\R)$, we have:
\[
  Z_{\sqrt{2}} g (x,\gamma) = 0 \quad \text{for } (x,\gamma) \in
  (\tfrac14 \Z \setminus \Z) \times (\Z +\tfrac12),
\]
and
\[
  Z_{\sqrt{3}} g (x,\gamma) = 0 \quad \text{for } (x,\gamma) \in
  (\tfrac13 \Z + \tfrac16) \times (\Z +\tfrac12).
\] \label{item:cor-zeros-1}
\item For $g \in E_3 \cap W(\R)$ and $s\in \{2,3,4\}$, we have:
\[
  Z_{\sqrt{s}} g (x,\gamma) = 0 \quad \text{for } (x,\gamma) \in \tfrac1s \Z \times \Z ,
\]
and 
\begin{equation}
\label{eq:new-zeros-Zak-at-12}
  Z_{\sqrt{2}} g (x,\gamma) = 0 \quad \text{for } (x,\gamma) \in
  (\tfrac12 \Z+\tfrac14) \times (\Z+\tfrac12) .
\end{equation}
\label{item:cor-zeros-2}
\end{enumerate}
\end{corollary}
\begin{proof}
    Theorem~\ref{thm:symm_modu_even_n} for $\ell=2$ states that
    the function
    \[
\R \ni \kappa  \mapsto Z_{s 2^{\kappa}}g\bigl(\frac{\tfrac{1}{2}+p}{s^{2}},\tfrac{1}{2}\bigr)
    \]
    is odd for
    $s^{2} \in \set{2,3}$ and $p \in \set{0,1,\dots,s-1}$.
Since the function is odd, taking $\kappa=0$ (i.e., $\lambda=1$) yields
$Z_{s}\bigl(\frac{1/2+p}{s^{2}},\tfrac{1}{2}\bigr)=0$. The
statement~(\ref{item:cor-zeros-1}) now follows by~\eqref{eq:Zak_even_3} and the
quasi-periodicity of the Zak transform.

Applying Theorem~\ref{thm:symm_modu_odd_n} and Theorem~\ref{thm:symm_modu_even_n} with $\ell=3$ will in the same way shows assertion (\ref{item:cor-zeros-2}).
  \end{proof}

\begin{remark}
\label{rem:four-zeros}
The zeros of the Zak transform of Hermite functions in \eqref{eq:Zak-1-zero-origo}, \eqref{eq:Zak-1-zero-mid} and Corollary~\ref{lem:extension-identities} coincide with all the new zeros found by Boon, Zak and Zucker and correspond to the filled circles in Figure~1 of \cite{WOS:A1983QC19800014}. All these zeros are associated with an \emph{odd} modular characteristic of the Zak transform as illustrated in the proof of Corollary~\ref{lem:extension-identities}.  The counterexamples to the frame set conjecture found in \cite{Lemvigsome2017} follow by a simple application of
Corollary~\ref{cor:obstr-frame-prop-even-and-odd} and Corollary~\ref{lem:extension-identities}.

In \cite{WOS:A1983QC19800014}, there are no non-trivial zeros of the Zak transform of Hermite functions $h_n$ of order $n=4m$ and $n=4m+1$, where ``trivial zeros'' refers to those resulting from the window function being even or odd. The analysis of zeros of the Zak transform of these Hermite functions will be the focus of the next section.

As a historical remark, let us mention that the zeros of the Zak transform in \eqref{eq:Zak-1-zero-origo}, \eqref{eq:Zak-1-zero-mid} and \eqref{eq:new-zeros-Zak-at-12} are not discussed in \cite{Lemvigsome2017} 
since they do not lead to new obstructions for the Gabor frame property. However, it is interesting to note that the second named author in \cite{Lemvigsome2017}, being unaware of the work in \cite{WOS:A1983QC19800014}, independently identified the same non-trivial zeros as in \cite{WOS:A1983QC19800014}. 
\end{remark}

\subsection{Additional zeros of the Zak transform as a function of the modular parameter}
\label{sec:additional-zeros-zak}

In this section we will take $\gamma_{0}=1/2$ for $h_{n}$ being an even function
($n$ even) and $\gamma_{0}=0$ for $h_{n}$ being an odd function ($n$ odd). We
will also let $x_{0}$ be a fixed, but arbitrary real number in
$\itvcc{-1/4}{1/4}+\Z$. The function
$\R_{>0} \ni \lambda \mapsto Z_{s\lambda}h_{n}(x_{0},\gamma_{0})$ is continuous
for any $n \in \Z_{\ge 0}$. We will here show that the function also has a zero
for any values of $n \ge 3$.

\begin{lemma}
  \label{lem:h_even_zero}
  Let $n \ge 4$ be an even integer. Suppose $x_0 \in \itvcc{-1/4}{1/4}+\Z$ is given. Then there exists a $\lambda > 0$
  so that $Z_\lambda h_n(x_0,\tfrac{1}{2})=0$ and therefore
  \begin{equation}
    \label{eq:Zak_h_even_zeros}
      Z_\lambda h_n(\pm x_0+k,\tfrac{1}{2}+\ell)=0 \quad \text{for all
  $k,\ell \in \Z$.}
  \end{equation}
\end{lemma}
\begin{proof}
  The Zak transform $Z_\lambda h_n(\cdot,1/2)$ is both
  quasi-periodic~\eqref{eq:Zak-quasi-periodicity} and
  symmetric~\eqref{eq:Zak_even_and_odd2} in the first variable, hence if
  $Z_\lambda h_n(x_0,\tfrac{1}{2})=0$, then~\eqref{eq:Zak_h_even_zeros} also
  holds. It suffices to show that $Z_\lambda h_n(x_0,\tfrac{1}{2})=0$ for
  $x_0 \in \itvoc{0}{1/4}$. The cases $x_0 \in \itvco{-1/4}{0}$ and $x_{0}=0$
  will follow by symmetry~\eqref{eq:Zak_even_and_odd2} and continuity,
  respectively.

 So, we assume $0 < x_0 \le 1/4$. We first show that $Z_\lambda h_n(x_0,\tfrac{1}{2})$ is
  positive for sufficiently large $\lambda>0$. First, we rewrite the
  series: 
\begin{align*}
  Z_\lambda h_n(x_0,\tfrac{1}{2}) &= \sum_{k\in\Z} (-1)^k h_n
  (\lambda(x_0+k)) \\
  &= \sum_{k=0}^\infty (-1)^k h_n (\lambda(x_0+k)) + \sum_{k=1}^\infty (-1)^k h_n (\lambda(x_0-k)) \\
&= \sum_{k=0}^\infty (-1)^k \bigl[h_n (\lambda(x_0+k)) - h_n(\lambda(1-x_0+k))\bigr]
\end{align*}
Now, pick $\lambda_1$ so that $\lambda_1 x_0 > \frac{1}{\sqrt{2 \pi}}
\sqrt{2n+1}$. 
Since $h_n$ is strictly convex on $\itvoo{\frac{1}{\sqrt{2 \pi}}
  \sqrt{2n+1}}{\infty}$, its derivative is monotonically
increasing to zero on the same interval. Note that $x_0+k<1-x_0+k$
since $x_0<1/2$.
Hence, by the mean value theorem, the sequence of positive numbers
\[ \seq{h_n (\lambda_1(x_0+k)) - h_n(\lambda_1(1-x_0+k))}_{k=0}^\infty \]
decreases monotonically to zero. It thereby follows that $Z_{\lambda_1}
h_n(x_0,\tfrac{1}{2}) > 0$
by the alternating series test. 

Let $\tilde{x}_1, \tilde{x}_2, \dots, \tilde{x}_{n/2}$ denote the positive zeros of $h_n$ in
descending order. Recall that these are related to the zeros $x_{k}$ of the Hermite
polynomial $H_{n}$ by $\tilde{x}_{k} = 1/\sqrt{2\pi} \, x_{k}$. Pick $\lambda_0$ so that
\[ \frac{n-2}{n-1} \tilde{x}_1 \le x_0 \lambda_0  < \tilde{x}_1 \]
We then claim that $Z_{\lambda_0} h_n(x_0,\tfrac{1}{2}) < 0$.
To see the claim, we first note that $h_n$ is negative on the interval
$\itvoo{\tilde{x}_2}{\tilde{x}_1}$.  Next, we rewrite the Zak transform as:
\begin{align}
  Z_{\lambda_0} h_n(x_0,\tfrac{1}{2})
  &= h_n(\lambda_0 x_0) + \sum_{k=1}^\infty (-1)^k h_n
    (\lambda_0(x_0+k)) +  \sum_{k=1}^\infty (-1)^k h_n
    (\lambda_0(x_0-k)) \nonumber \\
&= h_n(\lambda_0 x_0) + \sum_{k=1}^\infty (-1)^k h_n(\lambda_0(x_0+k)) +  \sum_{k=1}^\infty (-1)^k h_n(\lambda_0(-x_0+k))  \label{eq:n-even-Zak-negative-inf-series-split_1}
\end{align}

Suppose 
\begin{equation}
\lambda_0 (-x_0+1) \ge \frac{1}{\sqrt{2 \pi}}
\sqrt{2n+1}. \label{eq:lambda0-estimate_1}
\end{equation}
Then, since $x_0\ge 0$, we have $\lambda_0 (\pm x_0 +k) \ge \tfrac{1}{\sqrt{2 \pi}} 
\sqrt{2n+1}$ for $k \in \Z_{>0}$, and 
it follows again by convexity and positivity of $h_n$ on $\itvoo{\frac{1}{\sqrt{2 \pi}}
  \sqrt{2n+1}}{\infty}$ and the alternating series
test that the two series in
\eqref{eq:n-even-Zak-negative-inf-series-split_1} are negative as the
first term in both series is negative. Moreover, by convexity of the roots of the Hermite
polynomials, we have $\tilde{x}_2 \le \frac{n-2}{n-1} \tilde{x}_1$, and it follows
that also
$h_n(\lambda_0 x_0)$ is negative. Hence, to finish the proof of the claim, we only
have to show that~\eqref{eq:lambda0-estimate_1} holds. However,
by choice of $\lambda_0$, we have
\[
   \lambda_0 (-x_0+1) \ge (-1+1/x_0) \frac{n-2}{n-1} \tilde{x}_1 >
       \frac{1}{\sqrt{2 \pi}}  \frac{3^{3/2}}{2^{1/2}} \frac{n-2}{\sqrt{n+1}},
\]
where the last inequality follows by
Lemma~\ref{lem:x1-simple-bound} and by $x_0\le 1/4$. It is straightforward to verify
that $\tfrac{3^{3/2}}{2^{1/2}} \tfrac{n-2}{\sqrt{n+1}} \ge
\sqrt{2n+1}$ holds for $n \ge 4$. Thus, we conclude that
\eqref{eq:lambda0-estimate_1} holds for $n \ge 4$.
\end{proof}
                                                                                                                          
\begin{lemma}
    \label{lem:h_odd_zero}
  Let $n \ge 3$ be an odd integer. Suppose $x_0 \in \itvcc{-1/4}{1/4}+\Z$ is given. Then there exists a $\lambda > 0$
  so that $Z_\lambda h_n(x_0,0)=0$ and therefore
  \begin{equation*}
      Z_\lambda h_n(\pm x_0+k,\ell)=0 \quad \text{for all
  $k,\ell \in \Z$.}
  \end{equation*}
\end{lemma}
\begin{proof}
  As in the proof of Lemma~\ref{lem:h_even_zero}, it suffices to show that
  $Z_\lambda h_n(x_0,0)=0$ for $x_0 \in \itvoc{0}{1/4}$. So, we assume
  $0 < x_0 \le 1/4$. We first show that $Z_\lambda h_n(x_0,0)$ is positive for
  sufficiently large $\lambda>0$. Using that $h_n$ is an odd function, we
  rewrite the series:
\begin{align*}
  Z_\lambda h_n(x_0,0) &= \sum_{k\in\Z} h_n
  (\lambda(x_0+k)) \\
  &= \sum_{k=0}^\infty  h_n (\lambda(x_0+k)) + \sum_{k=1}^\infty h_n
    (\lambda(x_0-k)) \\
    &= \sum_{k=0}^\infty  h_n (\lambda(x_0+k)) - \sum_{k=0}^\infty h_n
      (-\lambda(x_0-k-1)) \\
    &= \sum_{k=0}^\infty  h_n (\lambda(x_0+k)) -  h_n (\lambda(1-x_0+k)).
\end{align*}
Pick $\lambda_1$ so that $\lambda_1 x_0 > \frac{1}{\sqrt{2 \pi}}
\sqrt{2n+1}$. Then $h_n$ is monotonically decreasing to zero on
$\itvoo{\lambda_1 x_0}{\infty}$, and it follows that $h_n
(\lambda(x_0+k)) -  h_n (\lambda(1-x_0+k)) > 0$ for all $k \in \Z_{\ge
0}$. We conclude that $Z_{\lambda_1} h_n(x_0,0) > 0$.

Let $\tilde{x}_1, \tilde{x}_2, \dots, \tilde{x}_{(n-1)/2}$ denote the positive
zeros of $h_n$ in descending order. Pick $\lambda_0$ so that
\[ \frac{n-2}{n-1} \tilde{x}_1 \le x_0 \lambda_0  < \tilde{x}_1 \]
We then claim that $Z_{\lambda_0} h_n(x_0,0) < 0$.
To see the claim, we first note that $h_n$ is negative on the interval
$\itvoo{\tilde{x}_2}{\tilde{x}_1}$.  Next, we rewrite the Zak transform as:
\begin{align}
  Z_{\lambda_0} h_n(x_0,0)
  &= h_n(\lambda_0 x_0) + \sum_{k=1}^\infty  h_n
    (\lambda_0(x_0+k)) +  h_n
    (\lambda_0(x_0-k)) \nonumber \\
&= h_n(\lambda_0 x_0) + \sum_{k=1}^\infty  h_n(\lambda_0(x_0+k)) - h_n(\lambda_0(-x_0+k))  \label{eq:n-even-Zak-negative-inf-series-split}
\end{align}

Note that $0<-x_0+k < x_0 + k$. Assuming
\begin{equation}
\lambda_0 (1-x_0) \ge \frac{1}{\sqrt{2 \pi}}
\sqrt{2n+1},\label{eq:lambda0-estimate}
\end{equation}
we have $\lambda_0 (\pm x_0 +k) \ge 1/\sqrt{2 \pi} \sqrt{2n+1}$ for
$k \in \Z_{>0}$. Since $h_n$ is monotonically decreasing to zero on
$\itvoo{\lambda_1 x_0}{\infty}$, it follows again by convexity and positivity of
$h_n$ on $\itvoo{\frac{1}{\sqrt{2 \pi}} \sqrt{2n+1}}{\infty}$ and the
alternating series test that the two series in
\eqref{eq:n-even-Zak-negative-inf-series-split} are negative as the first term
is negative. Moreover, by convexity of the roots of the Hermite polynomials, we
have $\tilde{x}_2 \le \frac{n-2}{n-1} \tilde{x}_1$, and it follows that also
$h_n(\lambda_0 x_0)$ is negative. Hence, to finish the proof of the claim, we
only have to show that~\eqref{eq:lambda0-estimate} holds. However, by choice of
$\lambda_0$, we have
\[
   \lambda_0 (1-x_0) \ge (1/x_0-1) \frac{n-2}{n-1} \tilde{x}_1 >
       \frac{1}{\sqrt{2 \pi}}  \frac{3^{3/2}}{2^{1/2}} \frac{n-2}{\sqrt{n+1}},
\]
where the last inequality follows by Lemma~\ref{lem:x1-simple-bound}. It is
straightforward to verify that
$\tfrac{3^{3/2}}{2^{1/2}} \tfrac{n-2}{\sqrt{n+1}} \ge \sqrt{2n+1}$ holds for
$n \ge 3$. Thus, we conclude that \eqref{eq:lambda0-estimate} holds for
$n \ge 3$.
\end{proof}

\section{Counterexamples}
\label{sec:counterexamples}

In the two next subsections we present counterexamples to the frame set
conjecture for Gabor systems $\mathcal{G}(h_{n},\alpha,\beta)$ for $n > 2$.
Recall that all counterexamples to the frame set conjecture for Gabor systems
$\mathcal{G}(h_{n},\alpha,\beta)$ will be found on hyperbolas
$\alpha \beta=1/2, \alpha \beta=1/3, \alpha \beta=1/4$ and $\alpha \beta=2/3$;
their precise location will, however, not be given. In
Section~\ref{sec:estim-locat-count}, we will show how to estimate the location
of the counterexamples. In the final subsection, Section~\ref{sec:numer-exper},
we will illustrate numerically that there are many more counterexamples than
what we prove the existence of. The Python code used in all the numerical
experiments is open-source and is hosted for public access on GitHub
at \url{https://github.com/jakoblem/gfsp}.

\subsection{Non-frame property for Hermite functions of even order}
\label{sec:non-frame-property-even}

We first present new counterexamples for even Hermite functions of degree four
or larger, that is, we will exhibit points $(\alpha, \beta)$, with
$\alpha\beta<1$ being rational, not belonging to the frame set
$\frameset{(h_{n})}$. So let $n\ge 4$ be an even integer.
  \begin{description}
    \item[On $\bm{\alpha\beta=1/2}$.] Lemma~\ref{lem:h_even_zero}
  with $x_{0}=1/4$ yields the existence of $\lambda_{1} > \sqrt{2}$ for which
  $Z_{\lambda_{1}}h_{n}(1/4,1/2)=0$. By
  Corollary~\ref{cor:obstr-frame-prop-even-and-odd}\eqref{item:cor-obsr-even} and Lemma~\ref{lem:frame-set-symmetry-h_n},
  we conclude that
  \begin{equation}
    (\lambda_{1}/2,1/\lambda_{1}),    (1/\lambda_{1},\lambda_{1}/2) \notin \frameset{(h_{n})}, \; \text{where
      $n \ge 4$ is even}.
    \label{eq:lambda_1-points}
 \end{equation}
  \item[On $\bm{\alpha\beta=p/3 \, (p=1,2)}$.] Lemma~\ref{lem:h_even_zero}
  with $x_{0}=1/6$ yields the existence of $\lambda_{2} > \sqrt{3}$ so that
  $Z_{\lambda_{2}}h_{n}(1/6,1/2)=0$. Then Theorem~\ref{thm:symm_modu_even_n} with $s^{2}=3$
  shows that also $Z_{3/\lambda_{2}}h_{n}(1/6,1/2)=0$. Hence, from
  Corollary~\ref{cor:obstr-frame-prop-even-and-odd}\eqref{item:cor-obsr-even}
  and Lemma~\ref{lem:frame-set-symmetry-h_n} we conclude that
  \begin{multline}
    (\lambda_{2}/3,1/\lambda_{2}), (2\lambda_{2}/3,1/\lambda_{2}), (1/\lambda_{2},\lambda_{2}/3), (1/\lambda_{2}, 2\lambda_{2}/3), \\ (2/\lambda_{2},\lambda_{2}/3), (\lambda_{2}/3, 2/\lambda_{2}) \notin \frameset{(h_{n})} ,\; \text{
      where $n \ge 4$ is even}. \label{eq:lambda_2-points}
  \end{multline}
  \end{description}

\subsection{Non-frame property for Hermite functions of odd order}
\label{sec:non-frame-property-odd}

We then turn to new counterexamples for odd Hermite functions of degree three or larger.
Let $n \ge 3$ be an odd integer.
\begin{description}
  \item[On $\bm{\alpha\beta=1/3}$.] Lemma~\ref{lem:h_odd_zero}
  with $x_{0}=1/6$ yields the existence of $\lambda_{3} > \sqrt{3}$ for which
  $Z_{\lambda_{3}}h_{n}(1/6,0)=0$. By
  Corollary~\ref{cor:obstr-frame-prop-even-and-odd}\eqref{item:cor-obsr-odd}and Lemma~\ref{lem:frame-set-symmetry-h_n}, it follows that
  \begin{equation}
    (\lambda_{3}/3,1/\lambda_{3}),    (1/\lambda_{3},\lambda_{3}/3) \notin \frameset{(h_{n})}, \; \text{
      where $n \ge 3$ is odd}.
    \label{eq:lambda_3-points}
 \end{equation}
  \item[On $\bm{\alpha\beta=1/4}$.] Lemma~\ref{lem:h_odd_zero}
  with $x_{0}=1/4$ yields the existence of $\lambda_{4} > 2$ for which
  $Z_{\lambda_{4}}h_{n}(1/4,0)=0$. Hence, from
  Corollary~\ref{cor:obstr-frame-prop-even-and-odd}\eqref{item:cor-obsr-odd} and Lemma~\ref{lem:frame-set-symmetry-h_n} we conclude that
  \begin{equation}
    (\lambda_{4}/4,1/\lambda_{4}), (1/\lambda_{4},\lambda_{4}/4)\notin \frameset{(h_{n})}, \; \text{where
      $n \ge 3$ is odd}.\label{eq:lambda_4-points}
  \end{equation}
\end{description}

\subsection{Bounds on the location of the counterexamples}
\label{sec:estim-locat-count}

The values of $\lambda_{i}, i=1,2,3,4$ in the counterexamples in equations \eqref{eq:lambda_1-points},
\eqref{eq:lambda_2-points}, \eqref{eq:lambda_3-points}, and
\eqref{eq:lambda_4-points} certainly depend on $n$. While we did not determine
the exact value of $\lambda_{i}$, we are, in fact, able to give lower and upper
bounds depending only on $n$. Following the proofs of
Lemma~\ref{lem:h_even_zero} and~\ref{lem:h_odd_zero}, we see that
\[
   x_{1} < \sqrt{2\pi} \, x_{0}\, \lambda_{i} < \sqrt{2n+1},
\]
where $x_{1}$ is the largest zero of the Hermite polynomial $H_{n}(x)$ and
$x_{0}$ is chosen in Lemma~\ref{lem:h_even_zero} and~\ref{lem:h_odd_zero}. Using
Lemma~\ref{lem:x1-simple-bound} we then arrive at
\begin{equation}
  \label{eq:lower-and-upper-bound-on-lambda}
   \frac{\sqrt{3/2}}{ x_{0}\sqrt{2\pi}} \frac{n-1}{\sqrt{n+1}} <\lambda_{i} < \frac{1}{x_{0}\sqrt{2\pi}} \sqrt{2n+1},
 \end{equation}
The bounds hold for \emph{even} $n\ge 4$ in case $i=1,2$ and for \emph{odd}
$n\ge 3$ in case $i=3,4$. Figure~\Ref{fig:lower_upper_bound} shows that, in
particular, the lower bound provides a good estimate of the true value of
$\lambda_{i}$.
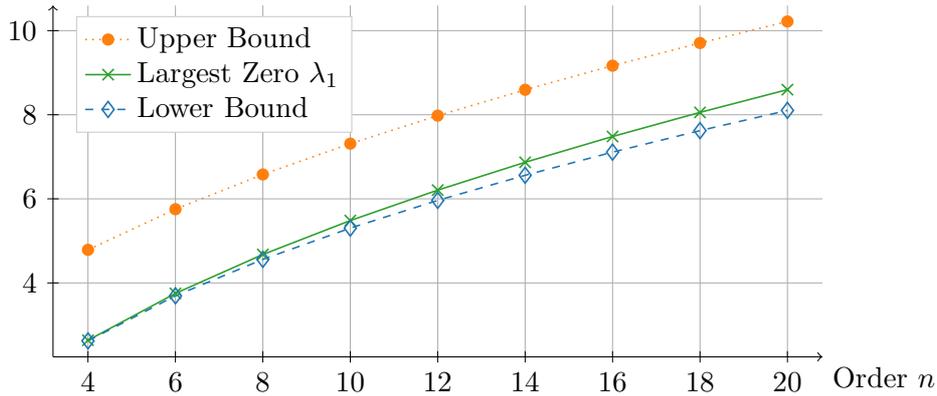
\begin{figure}
  \centering
\begin{tikzpicture}

\definecolor{darkgray176}{RGB}{176,176,176}
\definecolor{darkorange25512714}{RGB}{255,127,14}
\definecolor{forestgreen4416044}{RGB}{44,160,44}
\definecolor{lightgray204}{RGB}{204,204,204}
\definecolor{steelblue31119180}{RGB}{31,119,180}

\begin{axis}[
width=0.75\textwidth,
height=0.4\textwidth,
legend cell align={left},
legend style={
  fill opacity=0.8,
  draw opacity=1,
  text opacity=1,
  at={(0.03,0.97)},
  anchor=north west,
  draw=lightgray204
},
x grid style={darkgray176},
xlabel={Order $n$},
xmajorgrids,
xmin=3.2, xmax=20.8,
xtick style={color=black},
y grid style={darkgray176},
ymajorgrids,
ymin=2.24232664809701, ymax=10.5976975252246,
ytick style={color=black}
]
\addplot [semithick, darkorange25512714, dotted, mark=*, mark size=2, mark options={solid}]
table {%
4 4.78730736481719
6 5.75362739175159
8 6.57952464247954
10 7.31273279143145
12 7.97884560802865
14 8.59347971398312
16 9.16699568847508
18 9.70668461991024
20 10.2179079399006
};
\addlegendentry{Upper Bound}
\addplot [semithick, forestgreen4416044, mark=x, mark size=3, mark options={solid}]
table {%
4 2.63600358223107
6 3.75102328412274
8 4.67662070190268
10 5.48331661876521
12 6.20710288356161
14 6.86890611430174
16 7.48214481861465
18 8.05602339998763
20 8.59717564752081
};
\addlegendentry{Largest Zero $\lambda_{1}$}
\addplot [semithick, steelblue31119180, dashed, mark=diamond, mark size=3, mark options={solid}]
table {%
4 2.62211623342099
6 3.69348781844744
8 4.56029011109392
10 5.3034912118052
12 5.96261400303117
14 6.56014455725242
16 7.11021093712182
18 7.62233106106616
20 8.10325750767108
};
\addlegendentry{Lower Bound}
\end{axis}

\end{tikzpicture}

         \caption{Illustration of the lower and upper bounds in
   \eqref{eq:lower-and-upper-bound-on-lambda} for $i=1$ and $x_{0}=1/4$ as well
as the exact value of $\lambda_{1}$ (i.e., the largest value of $\lambda_{1}$ so that
$Z_{\lambda_{1}}h_{n}(1/4,1/2)=0$) for $n=4,6,\dots, 20$. The lower bound is a
direct consequence of Lemma~\ref{lem:x1-simple-bound}.}
\label{fig:lower_upper_bound}
\end{figure}

In \cite{Lemvigsome2017} the obstructions of the frame property of
$\mathcal{G}(h_{n},\alpha,\beta)$ for $n=2,3,6,7,\dots$ all occurred near
uniform sampling $\alpha=\beta$, in fact,
$\alpha,\beta \in \itvccs{1/2}{2/\sqrt{3}} \approx \itvcc{0.5}{1.15}$. The
bounds in \eqref{eq:lower-and-upper-bound-on-lambda} in combination with the
symmetry results in Section~\ref{sec:symmetries-par} show that for any $n\ge 3$
we have obstructions of the frame property of $\mathcal{G}(h_{n},\alpha,\beta)$
where $\alpha$ and $\beta$ grow (up to constants) as $n^{1/2}$ and $n^{-1/2}$
and vice versa.

\subsection{Numerical experiments}
\label{sec:numer-exper}

This paper concludes with a series of numerical experiments that illustrate the
complex characteristics of the Gabor frame sets of Hermite functions. The Python
code developed for these experiments is accessible on GitHub
at \url{https://github.com/jakoblem/gfsp}. We believe that this code will serve
as a valuable resource for researchers exploring the Gabor frame sets of
functions with even or odd symmetry. For those interested in replicating our
work, all the required files to perform the numerical experiments and to
recreate figures and tables found within this paper are also available on the
specified GitHub repository.

Let $g \in W(\R)$ be an even, continuous function. By
Corollary~\ref{cor:obstr-frame-prop-even-and-odd}\eqref{item:cor-obsr-even} any
zero of the function $\lambda \mapsto Z_{\lambda} g(x_{0},1/2)$ for $x_{0}=1/4$
or $x_{0}=1/6$ will correspond to a non-frame property of
$\mathcal{G}(g,\alpha,\beta)$ on the hyperbolas $\alpha \beta=1/2$ and
$\alpha \beta=p/3 \, (p=1,2)$, respectively. To be precise, the location is
determined by $\beta = 1/\lambda_{0}$, where $\lambda_{0}$ is a zero of
$\lambda \mapsto Z_{\lambda} g(x_{0},1/2)$. Part~\eqref{item:cor-obsr-odd} of
Corollary~\ref{cor:obstr-frame-prop-even-and-odd} can, similarly, be used to
prove the non-frame property of Gabor systems generated by odd functions, where
one is interested in zeros of $\lambda \mapsto Z_{\lambda} g(x_{0},0)$.

Figure~\ref{fig:thm_symmetry_even} in Section~\ref{sec:symmetries-par} shows the graph of
$\lambda \mapsto Z_{\lambda} h_{n}(1/6,1/2)$ and, thus, each zero corresponds to a non-frame
property on each of the hyperbolas $\alpha \beta=1/3$ and $\alpha \beta=2/3$.
Numerically we can easily find the zero of
$\kappa \mapsto Z_{\sqrt{3} 2^{\kappa}} h_{n}(1/6,1/2)$ using \texttt{fsolve} from,
e.g., Python's SciPy library.  For e.g., $n=8$, the zeros are:
\begin{equation*}
    \kappa_{\text{zeros}}=\begin{bmatrix}
        -2.01794767 & -1.45344028 & -0.67928838 &  0.67928838 &  1.45344028 &  2.01794767
    \end{bmatrix}
  \end{equation*}
which corresponds to the non-frame property of
$\mathcal{G}(h_{8},\alpha_{i},\beta_{i})$ for $i=1,\dots,6$ where
$\alpha_{i},\beta_{i}$ are given in Table~\ref{table:h8-nonframe-points}.
\begin{table}[h!]
\centering
\begin{tabular}{c |S[table-format=2.12] S[table-format=2.12] }
 \hline
  & $\alpha_{i}$ & $\beta_{i}$ \\
 \hline
  $i=1$ & 0.14255308  & 2.33831035  \\
  $i=2$ & 0.21081924   & 1.58113333 \\
  $i=3$ & 0.36053978   & 0.92453967  \\
  $i=4$ & 0.92453967   & 0.36053978  \\
  $i=5$ & 1.58113333   & 0.21081924 \\
  $i=6$ & 2.33831035  & 0.14255308 \\
\end{tabular}
\caption{Non-frame property of
$\mathcal{G}(h_{8},\alpha_{i},\beta_{i})$ for $i=1,\dots,6$, i.e.,
$(\alpha_{i},\beta_{i}) \notin \frameset{(h_{8})}$. Note that the points
$(2\alpha_{i},\beta_{i}), (\alpha_{i},2\beta_{i}),(\beta_{i},\alpha_{i}), (2\beta_{i}, \alpha_{i}), (\beta_{i}, 2\alpha_{i})  \notin \frameset{(h_{8})}$
also belong to the complement of $\frameset{(h_{8})}$.}
\label{table:h8-nonframe-points}
\end{table}

  Note that we numerically only need to find the positive zeros in the list
$\kappa_{\text{zeros}}$ since the function
$\kappa \mapsto Z_{\sqrt{3} 2^{\kappa}} h_{n}(1/6,1/2)$ is even by
Theorem~\Ref{thm:symm_modu_even_n}. More importantly, for $\kappa > 0$ we have
no issues with truncation errors that was apparent in
Figure~\ref{fig:thm_symmetry_even} for $\kappa < -4$, where the Zak transform was
approximated with a partial sum using 40 terms. Negative values of $\kappa$
correspond to small values of $\lambda$ in
$Z_{\lambda} h_{n}$~\eqref{eq:zakTransform} and will therefore eventually lead
to truncation errors, even for functions as $h_{n}$ with fast decay. For Hermite
functions this issue can be avoided using the symmetry results from Section~\ref{sec:symmetries-par}.

Returning to Figure~\ref{fig:thm_symmetry_even}, we see that the \emph{numbers} of
zeros of $\lambda \mapsto Z_{\lambda} h_{n}(1/6,1/2)$ increase with the order of
$n$. Counting the number of zeros on Figure~\ref{fig:thm_symmetry_even} we
conclude that, on each of the hyperbolas $\alpha \beta=p/3 \, (p=1,2)$, the
Gabor system $\mathcal{G}(h_{n},\alpha,\beta)$ fail to be a frame on at least
$1, 2, 5, 6$ locations for $n=2,4,6,8$, respectively. In the
Table~\ref{table:number-of-zeros-even} we
count\footnote{To be precise, we provide a Python function that automatically
  computes the number of zeros.} the number for zeros of
$\lambda \mapsto Z_{\lambda} h_{n}(x_{0},1/2)$ for all even orders below
$n < 22$ for $x_{0}=1/4$ and $x_{0}=1/6$, respectively. We see that the number
of zeros grows essentially as the order $n$ of the Hermite function.
\begin{table}[ht]
\centering
\begin{tabular}{l|*{11}{r}}
\hfill $n=$                & 0 & 2 & 4 & 6 & 8 & 10 & 12 & 14 & 16 & 18 & 20\\
\hline
$x_{0}=1/4, \gamma_{0}=1/2$ & 0 & 1 & 2 & 3 & 6 & 9 & 10 & 15 & 16 & 17 & 16 \\
$x_{0}=1/6, \gamma_{0}=1/2$ & 0 & 1 & 2 & 5 &  6 & 7 & 10  & 13 & 18 & 19 & 20 \\
\end{tabular}
\caption{Number of zeros of
  $\lambda \mapsto Z_{\lambda} h_{n}(x_{0},\gamma_{0})$ for Hermite functions of
  even order.}
\label{table:number-of-zeros-even}
\end{table}
Table~\ref{table:number-of-zeros-odd} shows a similar picture with the number for zeros of
$\lambda \mapsto Z_{\lambda} h_{n}(x_{0},0)$ for all odd orders below
$n < 22$ for $x_{0}=1/4$ and $x_{0}=1/6$, respectively.
\begin{table}[ht]
\centering
\begin{tabular}{l|*{11}{r}}
\hfill $n=$                & 1 & 3 & 5 & 7 & 9 & 11 & 13 & 15 & 17 & 19 & 21\\
\hline
$x_{0}=1/6, \gamma_{0}=0$ & 0 & 1 & 2 & 5 & 6 & 7 & 8 & 13 & 12 & 15 & 20 \\
$x_{0}=1/4, \gamma_{0}=0$ & 0 & 1 & 4 & 5 &  6 & 9 & 10  & 11 & 16 & 15 & 18 \\
\end{tabular}
\caption{Number of zeros of
  $\lambda \mapsto Z_{\lambda} h_{n}(x_{0},\gamma_{0})$  for Hermite functions of
  odd order.}
\label{table:number-of-zeros-odd}
\end{table}

For each of the zeros of found in Table~\ref{table:number-of-zeros-even}
and~\ref{table:number-of-zeros-odd} one can use symmetry properties as in
Sections~\ref{sec:non-frame-property-even} and \ref{sec:non-frame-property-odd},
respectively,
to extend the number of $(\alpha,\beta)$-points belonging to the complement of
the frame set $\frameset{(h_{n})}$.

\section*{Acknowledgements}
The authors would like to thank the reviewers for their valuable feedback and insightful suggestions. We are especially grateful to one reviewer for drawing our attention to the important references \cite{WOS:A1983QC19800014,WOS:A1981LU63700026}


\end{document}